\documentclass[12pt]{article}
\usepackage[T2A]{fontenc}
\usepackage[utf8]{inputenc}
\usepackage{amsmath}
\usepackage{cite}
\usepackage{amsmath,}
\usepackage{amssymb,amsfonts}
\usepackage{algorithmic}
\usepackage{textcomp}
\usepackage{amsthm}

\newtheorem{thm}{Theorem}
\newtheorem{cor}[thm]{Corollary}
\newtheorem{lem}[thm]{Lemma}
\newtheorem{prop}{Proposition}
\newtheorem{exm}[thm]{Example}
\newtheorem{defn}{Definition}
\newtheorem{rem}[thm]{Remark}

\begin{document}
\begin{center}
	{\LARGE\LARGE \bf Square root of an element in $PSL_2(\mathbb{F}_p)$, $SL_2(\mathbb{F}_p)$, $GL_2(\mathbb{F}_p)$ and $A_n$. Verbal width by set of squares in alternating group $A_n$ and Mathieu groups}

\medskip

	{\bf Skuratovskii Ruslan \textsuperscript{[0000-0002-5692-6123]} $^1$ }

\smallskip


	


\end{center}	
\bigskip

\section{Abstract }
The problems of square root from group element existing in $A_n$, $SL_2(F_p)$, $PSL_2(F_p)$ and $GL_2(F_p)$ for arbitrary prime $p$ are solved in this paper. The similar goal of root finding was reached in the GM algorithm adjoining an $n$-th root of a generator \cite{GM}
results in a discrete group
 for group $PSL(2,R)$, but we consider this question over finite field $F_p$. Well known the Cayley-Hamilton method \cite{Pell} for computing the square roots of the matrix $M^n$ can give answer of square roots existing over finite field only after computation of $det M^n$ and some real Pell-Lucas numbers by using Bine formula. Over method gives answer about existing $\sqrt{ M^n}$ without exponents $M$ to $n$-th power. We use only trace of $M$ or only eigenvalues of $M$.

In \cite{Amit} only the Anisotropic case of  group  $SL_1(Q)$,  where $Q$ is a quaternion division algebra over $k$  was considered. The authors of \cite{Amit} considered criterion to be square only for case $F_p$ is a field of characteristic not equal 2. We solve this problem even for fields $F_2$ and $F_{2^n}$. The criterion to $g \in SL_2 (F_2)$ be square in $SL_2(F_2)$ was not found by them what was declared in a separate sentence in \cite{Amit}.
Also in \cite{Amit} the split case  $G = SL_2(k)$, where under group splitting authors mean Bruhat decomposition is the double coset decomposition of the group $SL_2(k)$ with respect to the subgroup Borel $B$, consisting of upper triangualar matrix from $SL_2(k)$.

In this paper questions of the width in verbal subgroups of $A_n$ is considered.
Verbal width by set of squares of alternating group is found, it is happened equals to 2. The criterion of squareness in $A_n$ is presented.
The necessary and sufficient conditions when an element of alternating group $g A_n$ and $GL_2(F_p)$ as well  as for $SL_2(F_p)$ can be presented as a squares of one element are also found by us. Some necessary conditions to an element $g\in A_n$ being the square in $A_n$ are investigated. The criterion square root of an element existing in $PSL_2(\mathbb{F}_p)$ is found. The criterion of  existing an element square root in $PSL_2(\mathbb{F}_p)$ is found.

In this case, the verbal set turned out to be a system of generators, so we studied it and considered the diameter of the entire group.


\textbf{Key words}: verbal subgroup, verbal width, equation in matrix group, verbal length, set of squares of alternating group, criterion of  square root existing in $PSL_2(\mathbb{F}_p)$.\\
\textbf{2000 AMS subject classifications}:  20B27, 20E08, 20B22, 20B35, 20F65, 20B07.

\section{Introduction }
The characterization of verbal subgroups in a group is an interesting and
difficult problem.
In finite groups one of the earliest questions on verbal width comes from  
the Ore's paper \cite{O.Ore}, where O. Ore asked if the commutator length 
of arbitrary element in a non-abelian finite simple group is equal to 1.
Recently this hypothesis  was established \cite{Ore}.
Recently M. Liebeck, E. O'Brien, A. Shalev and Ph. Tiep \cite{Liebeck} proved that verbal width by set of squares  of a simple group is 2. We researched this question more deep for the alternating group $A_n$ and find a criterion to be square for an element of  $A_n$.
Also such criterion for $SL_2(F_p)$, $PSL_2(F_p)$ and $GL_2(F_p)$ were established.
This criterion is a stricter version of the formulated question for group extensions how large must an overgroup of a given group be in order to contain a square root of any element of the initial group $G$, which was considered in the paper of Anton A. Klyachko and D. V. Baranov \cite{Klyach}.

In any group $G$, the commutator of two elements $g$ and $h$ is a product of three
squares, namely $gh{{g}^{\text{-1}}}{{h}^{\text{-1}}}\text{ = }{{g}^{\text{2}}}{{\left( {{g}^{\text{-1}}}h \right)}^{\text{2}}}{{h}^{-2}}$.

The Verbal Width of the commutator of a free group was investigated by Lyndon and Newman \cite{Lynd} and further by Sucharit Sarkar \cite{Sar}. They proved that an arbitrary commutator of free group of rank greater than 1 cannot be generated by only 2 squares. The conditions when a commutator can be presented as a product of 2 squares were further found by him too \cite{Sar}.
The commutator subgroup of Sylow 2-subgroups of an alternating group have been previously investigated by the author \cite{SkCommEur}.


Taking in consideration the Sarkar's result we recall that
an every element of alternating group $A_n$ can be presented as a commutator of a symmetric group elements of the same degree $n$. Symmetric group $S_n$ is free group in class of permutation group. Therefore a natural question arises how many squares from $A_n$ as well as from $S_n$ one have to multiply to express such a commutator.

One of interesting algorithmic problem of combinatorial group theory was solved by Roman'kov \cite{Rom}.
 It was problem of determining for any element $g\in G$ is $g$ a commutator for free nilpotent group $N_r$ of arbitrary rank $r$ with class of nilpotency 2 \cite{Rom}.
The analogous problem can be formulated for a verbal set of squares.

The commutator width of perfect groups was completely studied by N. Nikolov \cite{Nik}.
In the work of author \cite{SkCommEur} it was proved  that the commutator width  of Sylows $2$-subgroups of ${{A}_{n}}$ is 1.

Through the ubiquity of group actions and the concrete representations which they afford, both finite and infinite permutation groups arise in many parts of mathematics and continue to be a lively topic of research in their own right. It then develops the combinatorial and group theoretic structure of primitive groups leading to the proof of the pivotal O'Nan-Scott Theorem which links finite primitive groups with finite simple groups.\\





\section{Preliminary}
The group is said to be \textbf{verbally simple} if it has no proper verbal
subgroups.

Considering multiplication of permutations
we utilize composing from left  to right,  as seems to be the practice in
introductory texts. However, granting the fact, the identity
$(1 2 3 ...  k) = (1 k) ... (1 3)(1 2)$
shows that that a $k$-cycle is even if and only if $k$ is odd.

Recall the Carmichael generator set $\langle s_n, t \rangle$:
$$
 s_n=\begin{cases}
(1,2)(3, \ldots, n),&\text{ $n\equiv 0 (mod 2)$},\\
(3,4, \ldots, n),     &\;\;\text{$n\equiv 1 (mod 2)$ }.\\
\end{cases}
$$
As well known \cite{CONRAD, Carm, Carm2}, the set of $t=(1,2,3)$ and $s_n$ is generating set for $A_{n}$, $n>3$.

We observe additional generating set of $A_n$ \cite{aviv} with 2 generators but which is not Carmichaecl
 generating set. In case of $n\equiv 1\left ( \bmod 2 \right)$ our generators are following:  $t=(1,2,3)$  and $s=(1,2,...,n)$, for even  $n$ the generators have form of cycles $t=(1, 2, 3)$ and $s=(2,3, … ,n)$. Consequantly the cycles number  $cyc(s)$ \cite{aviv} in the these generators $s, \, t$ are minimal.

 Following  Mitsuhashi's generating set is the nearest to Coxeter for $S_n$: $a_{ij}:= s_1s_{i} = (1,2)(i,i+1)$, $(2 \leq  i < j \leq n-1)$ \cite{aviv}.

 As well known the set $T$ of transpositions  $T= \{(1,2); (2,3); ... (n-1, n)   \} $ is generating set of involutive type for $S_n$.


Let $G$ be a group and $W$ be a set of words.

The verbal subgroup $V_W (G)$ of group $G$ is the subgroup generated by all
values of the words from $W$ in group $G$.

The set $W[G]$ generates
the verbal subgroup $V_W(G)$.


The  \textbf{width of the verbal subgroups} $ V_W (G)$
is equal to а least value of  $m\in \text{N}\cup \{\infty \}$ such that every element of the subgroup  $V_W\left( G \right)$  is represented as the  product of at most $m$  values of words  $V$.

The set of all \textbf{squares of elements} of $G$ is denoted by $\mathbb{S(G)}$.

The  \textbf{ diameter} of  $G$,  ${diam (G, S)}$  by set of  $S$  is defined to be the least
integer $n$, such that every element of  $G$ is a product of at most  $n$  elements from $S$ if such an integer $n$  exists, and  $diam(G,S) = \infty$  otherwise.

Another words \textbf{diameter} of $G$ is the maximum over $g \in G$ of the
length of the shortest expression of $g$ as a product of generators and their inverses.

For instance, if $G$ is cyclic of order
$N$ and $|S| = 1$ then $diam(F(G, S)) = [N/2]$.

The \textbf{Verbal width  of an element $g$ } over $V(G)$  $width(g, V_W(G))$ is such minimal $n$ that an element $g$ can be presented in form of a product of $n$ elements from $V_W(G)$.

The conditions when an arbitrary $g\in A_n$ can be presented as one squares were also found by us.


The Verbal Width of the commutator of a free group was investigated by Lyndon and Newman \cite{Lynd} and further by Sucharit Sarkar \cite{Sar}. They proved that an arbitrary commutator of free group of rank greater than 1 cannot be generated by only 2 squares. The conditions when a commutator can be presented as a product of 2 squares were further found by him too \cite{Sar}.
The commutator subgroup of Sylow 2-subgroups of an alternating group have been previously investigated by the author \cite{SkCommEur}.

We consider a similar problem in the commutator of the symmetric group $S_n$ that is alternating group $A_n$. Upon taking into account that the commutator subgroup of $S_n$ is the alternating group $A_n$, when $n>4$ \cite{Dix}, then the problem can be reformulated in terms of the alternating group.

Therefore, we research the verbal width by squares of $A_n$.

In a group $G$ a set of squares of its elements is denoted by $S(G)$.

Since $A_n$ is generated by all pairs of transpositions in particular even by mentioned above its Mitsuhashi's generating set that are included in $S(A_n)$ as a subset then $S(A_n)$ generates whole $A_n$ so one can consider $diam (A_n, S(A_n))$.

\begin{defn} We will call \textbf{mutually coprime cycles} that ones, which do not have the same elements.
\end{defn}
We denote cycles in cyclic presentation of $\pi$ by $C_{i}(\pi)$.

\begin{defn}
  The minimal number of transpositions in factorization of a permutation $\pi$ on transposition we will denote by $rnk(\pi)$. We set $rnk(e) =0$.
\end{defn}


For example, independent transpositions from the Moore generating system \cite{aviv} for $A_n$ can be used for factorization of $g$ in product of transpositions.
Note that presentation of $g\in {{A}_{n}}$ always has even number of transpositions.

 Consider the interaction of two interlinked permutations represented by squares.
 For this first note that if we have two special types of permutations $\pi_{1}, \pi_{2} \in S_{n}$ that decomposed into the product of transpositions, of $S_{n}$ then the following formula holds: 
\begin{equation}
    rnk(\pi_{1} \cdot \pi_{2}) = rnk(\pi_{1}) + rnk(\pi_{2}) -
2m, m \in \emph{N}.
\end{equation}

Note that addant '-2m' appears due to possible simplifying  of transpositions. 

The number $n-k(\pi )$ where  $k(\pi )$ is the number of cycles in decomposition of $\pi $  is  denoted  by $dec(\pi )$ i.e. decrement of permutation $\pi $.
As well known \cite{Sachk}  that the rank of $\pi $  is equal to $dec(\pi )$. Therefore the rank of $n$-cycle is  $n-1$.

Note, that a rank of permutation's product $\pi_{1} \pi_{2}$
  can be lesser then rank of each permutation $\pi_{1}, \pi_{2}$.

Let $k(\pi )$ be the number of cycles in decomposition of permutation $\pi $ of degree $n$.

The number $n-k(\pi )$ is  denoted  by $dec(\pi )$ and calls a decrement of permutation $\pi $.
As well known \cite{Sachk}  that the rank of $\pi $  is equal to $dec(\pi )$. Therefore the rank of $n$-cycle is $n-1$.

\begin{defn}
 The set of cycles $C_{1}(\pi), \
C_{2}(\pi),  \ldots , C_{m}(\pi), $ from $A_{n}, n
\geqslant 5$  satisfying the following condition $$ \sum_{i=1}^{m}
rnk(C_{i}(\pi)) = 2k, k \in \mathbb{N},$$  we will call
set of even rank. 
\end{defn}

\section{ Main result}

Let $s ^ {i} (G)$ be the image of $G$ after $i$ applications of squaring of all elements from $G$.
This mapping $s ^ {i} (G)$ keeps conjugacy classes in a alternating group. Let $h{{g}_{1}}{{h}^{-1}}={{g}_{2}}$  the its image under mapping $s ^ {1}(g_1)$ is ${{\left( h{{g}_{1}}{{h}^{-1}} \right)}^{2}}=hg_{1}^{2}{{h}^{-1}}$ is again conjugated to $hg_{2}^{2}{{h}^{-1}}$.

We have a decreasing sequence of subsets
$G \supset s (G) \supset s (s (G)) = s ^ 2 (G) \supset s (s (s (G)) = s ^ 3(G) \supset ... $
due to the finiteness of the group, this sequence on some number $i$ will no longer decrease (it will stabilize).

To understand the nature of this mapping, we will prove more precisely several statements that take into account the cyclic structure of elements from $A_n$.

Let us consider even permutation $g$ presented by odd $k$-cycle, where $k \equiv 1 (mod 2)$. Since $GCD(2,k)=1$ then there exists number $m$ such that
\begin{align*}
2\cdot m \equiv 1 \pmod k.
\end{align*}

Then
\begin{align*}
(g^m)^2 = g^{2m} = g.
\end{align*}
Therefore $g$ is square of an element of form $g^m$ for correspondent $m$.


\begin{lem}\label{odd}
 An element $g$ presented by odd $l$-cycle have invariant length under squaring.
\end{lem}
\begin{proof}
Since $ord(g^k)=\frac{ord(g)}{GCD(k, ord(g))}$ and in this case $k=2$, $ord(g)=l$, so $GCD(k, l)=1$ then $ord(g^2)$ is equal to $ord(g)$. Thus, the length of such $l$-cycles exposed to second power is $l$.
\end{proof}






\begin{lem}\label{root}
The square root from odd $k$-cycle $g$ is equal to $\frac{k+1}{2}$ power of $g$.
\end{lem}

\begin{proof}
As it was already mentioned in the Lemma \ref{odd},
 the length of an odd $k$-cycle remains invariant then we have that square root has the same length.
Let $g$ is presented as one $k$-cycle, $k\equiv 1\left(\bmod 2\right)$ then to find a square root of $g$ it is enough to compute ${{g}^{\frac{k+1}{2}}}$.
Indeed taking into consideration that order of $g$ is $k$ then after squaring of ${{g}^{\frac{k+1}{2}}}$ we have
${{g}^{k+1\left( \bmod k \right)}}\equiv g$.
Besides presentation of $k$-cycle can be starting from any of $k$ elements.
\end{proof}


\begin{lem}\label{2l}
 Squaring of cycle of degree $2l$ generates unique pair of $l$-cycles.
\end{lem}

\begin{proof}
Since when squaring a cycle of the form $\tau =\left( {{g}_{1}}...{{g}_{2l}} \right)=\left( {{a}_{1}}{{b}_{1}}{{a}_{2}}{{b}_{2}}\,...\,\,{{a}_{l}}{{b}_{l}} \right)$
an element ${{a}_{1}}$ mapping into ${{a}_{2}}$ analogously ${{a}_{i}}$ mapps into ${{a}_{i+1}}$, where $0<i<l-1$, in this case, the $l$-cycle is closed under the mapping ${{a}_{l}}$ to ${{a}_{1}}$. Therefore, we get a cycle 2 times shorter length, viz $l$. Similarly, element ${{b}_{i}}$ maps to ${{b}_{i+1}}$ and the second $l$-cycle forms. Due to the uniqueness of multiplication, there can be no other cyclic structure in ${{\tau }^{2}}$.
Note that, $l$ may already be as an odd number, i.e. such $l$-cycle generates an even substitution, and as an even number. But in the result of multiplication of the obtained $l$-cycles an even permutation is generated.
\end{proof}

\begin{prop}\label{necessaryCondofSquare}
 {\sl  Let $g\in {{A}_{n}}$. If $g$ can be presented in the form of a product of mutually coprime cycles containing an odd number of $2k$-cycles, then $g$ cannot be generated in the form ${{h}^{2}}{{t}^{2}}$, where $h, t \in A_n$ are mutually coprime. In the important particular case where $g=h^2$ we have $t=e$. Furthermore, the same holds even if $h,t\in {{S}_{n}}$.}
\end{prop}

\begin{proof} Let us assume that $g = {{h}^{2}}{{t}^{2}},\,\,\,h,t\in {{A}_{n}}$ and $g$ contains odd number of cycles of degree $2k_i$.
As it is shown in Lemma \ref{2l} the square of one $2{{l}_{i}}$-cycle is product of two ${{l}_{i}}$-cycles. Therefore, mutually coprime $2{{l}_{i}}$-cycles after squaring gives only pairs of mutually coprime ${{l}_{i}}$-cycles, where ${{l}_{i}}$ can be as odd as well as even.
We take into consideration only those of them where $l_i$ is even, precisely $l=2k_i$.
 Another way of appearing $2k_i$-cycle after squaring of mutually coprime cycles according to Lemma \ref{2l} does not exist.

  So number of cycles of even degree in ${{h}^{2}}{{t}^{2}}$ as well as in ${{h}^{2}}$ can be only even, thus the contradiction is found.
\end{proof}

\begin{exm}
The element $g=(1,3)(2,4)(5,7)(6, 8, 10)(9, 11, 12)$ is not product of mutually coprime elements $h,t\in S({A}_{12})$.
\end{exm}

\begin{cor}
An element with asymmetric cyclic structure $[{{m}^{1}},{{l}^{1}}]$, where $m\,\ne l$, $m=2t, l=2s$, can not generated as a product of squares of mutually coprime cycles. Example of such cycle is $({{c}_{1}}...{{c}_{m}})({{b}_{1}}...{{b}_{l}}),\,\,\,l\ne m$.
\end{cor}

We consider the squareness problem.
Does it exist algorithm that determining for any element $g\in G$ whether $g\in S(G)$ or another words whether width $ width(g, S(G))$=1.

Analogous problem of commutator verbal subgroup for free nilpotent group $N_r$ of arbitrary rank $r$ was solved by Roman'kov \cite{Rom}.

To describe necessary and sufficient conditions when an element $g \in A_n$ has verbal length by squares $S(A_n)$ equal to 1 and  to indicate all obstructions
to an element $g$ being the square in $A_n$ we introduce some new notations and deduce the criterion.

We denote by ${{p}_{l}}$ the number of pairs of $l$-cycles in cyclic presentation of $g$. We define fixator of a permutation $g$ (and denoted by $|Fix\left( g \right)| >1$) as the number of element of $\{1,...,n\}$, which are fixed under action of $g$ or power of  the fixator of $g$.
 We define $m_i$ as the number of cycles of length $i$ in cyclic structure. Note that $p_i=[\frac{m_i}{2}]$.
The following statement is criterion of square root existing.

\begin{thm} \label{Criterion}
 {\sl If in cyclic structure of $g \in {{A}_{n}}$ every even cycle appears an even number of times, i.e.
${{m}_{2k}}\equiv 0\left( \bmod \,2 \right)$,
and at least one of the two following conditions holds:}

1)  $$\left[
\begin{matrix}
	 |Fix\left( g \right)|>1 & (a)\\
	 \underset{k\in N}{\mathop{\max }}\,\left( {{m}_{2k-1}} \right)>1, & (b)
\end{matrix}
\right.$$

2)  $\sum\limits_{l=1}^{L}{{{p}_{2l}}}\equiv 0\left( \bmod \,2 \right),$

then this $g$ can be presented as $g={{h}^{2}}$,  $h\in {{A}_{n}}$. The vice versa is also true.

If we express the condition for the existence of a square root over $ S_n$ then condition 2) became to be a more simple: $m_{2l} \equiv 0 (mod 2)$ for all $l$ and the condition 1) is reduced.
\end{thm}

\begin{proof}
Since a square of each even $2l$-cycle of $h$ consists of two $l_i$-cycles by Lemma \ref{2l} then if a number of all $l_i$-cycles in a cyclic decomposition of $g=h^2$ is multiple of 4 then a square root $h$ contains two $2l_i$-cycles, as a result $h$ is an even permutation. In other words if a number of pairs $p_{l_i}$ in $g$ is even then square root of $g$ is even too.

 Indeed, one even cycle realizes an odd permutation. Thus, such element is not from $A_n$. But if a just mentioned $g$ has even number of pairs $p_i$ consisting of $l_i$ cycles for each $l_i$, then $h$ has even number of $2l_i$-cycles. So, $h$ realizes an even permutation. Therefore in this case $h$ belongs to $A_n$. Hence a square of such element $h$ have to contain an even number $L$ of pairs of $l_i$ cycles. Thus, if condition 2) $\sum\limits_{l=1}^{L}{{{p}_{2l}}}\equiv 0\left( \bmod \,2 \right)$ of this Theorem holds, then the square root as one element of $A_n$ exists.

 Let us show that if neither condition 2) nor the requirement from subparagraph a) of condition 1) is satisfied, then the fulfillment of item b) gives such square roots, one of which is an even permutation.
 Assume that $g$ has odd number of $p_i$ then one of square root $h$ is odd permutation. But if cyclic decomposition of $g$ satisfy the condition of item b): $\underset{k\in N}{\mathop{\max }}\,\left( {{m}_{2k-1}} \right)>1$, then the pair of $2k-1$-cycles exists in $g$.
  Taking in consideration that two odd ${(2k-1)}$-cycles is square of one even $(4k-2)$-cycle, therefore if $g$ contains more than ${(2k-1)}$-cycle
 then we can split off two of them and represent them as a square from one $(4k-2)$-cycle that presents odd permutation. This odd permutation this odd substitution complements the already existing odd permutation as a factor.

 As a result, this odd permutation $\pi$ complements the already existing odd permutation $\tau$ and  the product $\tau \cdot \pi$ became to be even permutation.
\end{proof}

\begin{cor}
The permutation $\pi \in S_n$ is square in $S_n$ iff ${{{m}_{2l}}}\equiv 0\left( \bmod \,2 \right)$ for each $l \in \mathbb{N}$.
\end{cor}

The following example illustrates fulfillment of condition (b) of just formulated Theorem \ref{Criterion}.
\begin{exm}
 The element $g=(1,3)(2,4)(5,7,9)(6, 8, 10)(11, 13, 12)$ satisfying condition (b) of previous Theorem \ref{Criterion} is square of one element in alternating group of minimal possible degree for this permutation, that is 13. Then square root of $g$ is $h=(1,2,3,4)(5,6,7,8,9,10)(11, 12, 13)$ (viz $h^2=g$). Note that $h$ contains precisely 2 cycles of even degree therefore $h$ realizes even permutation.
Another permutation satisfying this condition (b) is $(1,2)(3,4)(5,6,7)(8,9,10)$ here is $m_3 =2$.
 \end{exm}

\begin{exm}
 The element
$g= (10,11,12)(1,2,3)(4,5)(6,7)$ is square in $A_{12}$. Indeed, consider $h=(10,12,11)(1,3,2) (4,6,5,7)(8,9)$, then $h^2=(10,11,12)(1,2,3) (4,5)(6,7)$. The cyclic structure of the element $g$ satisfies to conditions a) and b) of Theorem \ref{Criterion}.
Note that also another root of $g$ exists $g= (10,11,12)(1,2,3)(4,5)(6,7) = ((10,1,11,2,12,3)(4,6,5,7))^2$.
\end{exm}

\begin{exm}
It is easy to check that the element
$$g=(1,2)(3,4)(5,6,7,8)(9,10,11,12)$$
 is square in $A_{12}$ too. This confirms criterion of existing square root formulated in Theorem \ref{Criterion} because cyclic structure of $g$ satisfies item 2.
\end{exm}

\begin{exm}
The following element $(1, 2 )(3, 4, 5, 6)$ can not be presented as square of $g\in {{A}_{6}}$. The element $(1,2,3)(4,5)(6,7)$ also is square neither over $A_7$ nor over $A_8$ but this element is square over $S_7$.
\end{exm}

{\bf Example.} The element $g= (1,2)(3,4)(5,6)(7,8)$ is square in $A_8$, in spite of the fact that this permutation $g$ has no fixed points.






{\bf Property }\label{closedness}
{\sl The set $S(A_n)$ is not a subgroup in $A_n$. Further, $A_n$ is verbally simple in particular the normal closure of $S(A_n)$ is $A_n$.}

\begin{proof}
If $S(A_n)$ was a subgroup in $A_n$ or normal closure formed a proper subgroup in $A_n$ then it would be normal subgroup in $A_n$ which is contradiction because of $A_n$ is simple group, when $n>4$. In other words, we can claim.
Since any verbal subgroup of any group is a fully invariant subgroup \cite{Chekh} then $V_{S(A_n)}(A_n)=A_n$. Because of every fully invariant subgroup is a normal \cite{Chekh, ChekhDanch}. Thus, the normal closure of $S(A_n)$ is $A_n$.

Let us show algebraic non-closedness of $S(A_n)$ under multiplication.
 Indeed, product of two squares $(1,2,3)(2,3,4) = (1,3)(2,4)$ is an even permutation but is not a square in $A_4$ because of $(1,3)(2,4)$ has unique root in $S_4$ which is $(1,2,3,4)$. But $(1,2,3,4)$ is an odd permutation, therefore  $(1,3)(2,4)\notin S(A_n)$ while $(1,2,3)$ and $(2,3,4)$ are squares in $A_4$.

The next example is over $A_{10}$:
 \begin{align*}
  & {{g}_{1}}=\left( 1,3,2 \right)\left( 4,6,5 \right)\left( 7,8 \right)\left( 9,10 \right) \\
 & {{g}_{2}}=\left( 1,3,2 \right) \\
 & \left( 1,3,2 \right)\left( 4,6,5 \right)\left( 7,8 \right)\left( 9,10 \right) \times \left( 1,3,2 \right)=\left( 4,6,5 \right)\left( 7,8 \right)\left( 9,10 \right)=\pi
\end{align*}
Note that $g_1, g_2 \in S(A_{10})$ but $\pi \notin S(A_{10})$. So the set of all squares $S(A_n)$ from $A_n$ does not coincide with the whole ${{A}_{n}}$. More interesting example is the following:
\begin{align*}
  & {{g}_{1}}=\left( 1,3,5 \right)\left( 4,6 \right)\left( 7,8 \right)\in S\left( {{A}_{10}} \right) \\
 & {{g}_{2}}=\left( 6,7 \right)\left( 9,10 \right)\in S\left( {{A}_{10}} \right) \\
 & {{g}_{1}}{{g}_{2}}=\left( 1,3,5 \right)\left( 6,4,7,8 \right)\left( 9,10 \right)\notin S\left( {{A}_{10}} \right)
\end{align*}
\end{proof}



\begin{prop}\label{StrofRoot}
Let the element $\pi =\left( {{a}_{1}}\,...\,{{a}_{l}} \right)\left( {{b}_{1}}\,...\,\,{{b}_{l}} \right) \in {{A}_{2l}}$ then the square roots from $\pi $ exist in ${A}_{2l}$. Moreover, there are $l+1$ different roots from $\pi$ when $l$ is odd and there are $l$ different roots if $l$ is even. If $l\equiv 1 (mod 2)$ then cyclic structure of a root can be either $[l^2]$ or $[2l]$.
\end{prop}


\begin{proof}
Let $\pi=(a_1...a_l )(b_1...b_l )$, then an element of the form $\tau =\left( {{g}_{1}}...{{g}_{2l}} \right)=\left( {{a}_{1}}{{b}_{1}}{{a}_{2}}{{b}_{2}}\,...\,\,{{a}_{l}}{{b}_{l}} \right)$ is square root of $\tau$. After being squared, it will give the product of two $l$-cycles, i.e. element of the form $\pi =\left( {{a}_{1}}\,...\,{{a}_{l}} \right)\left( {{b}_{1}}\,...\,\,{{b}_{l}} \right)$. Actually, the required element is $(a_1 b_1 a_2 b_2...a_l b_l )^2=(a_1...a_l )(b_1...b_l )$.
In more detail, when we square a cycle of the form $\tau =\left( {{g}_{1}}...{{g}_{2l}} \right)=\left( {{a}_{1}}{{b}_{1}}{{a}_{2}}{{b}_{2}}\,...\,\,{{a}_{l}}{{b}_{l}} \right)$,
an element ${{a}_{1}}$ maps into ${{a}_{2}}$,
 analogously ${{a}_{i}}$ maps into ${{a}_{i+1}}$, where $0<i<l-1$, at last, the $l$-cycle is closed under the mapping
 ${{a}_{l}}$ to ${{a}_{1}}$. Therefore we get a cycle of 2 times shorter length, viz $l$. Similarly, element ${{b}_{i}}$ maps to ${{b}_{i+1}}$ and the second $l$-cycle forms. Due to the uniqueness of multiplication, there can be no other cyclic structure of $\sqrt{\tau }$.

We can notice, that each of the $l$-cycles can be of odd or even length. However, their product is always even.
Besides, presentation of $l$-cycle can be starting from any of a $l$ elements.
Each shift of a subsequence $\{{{b}_{j}}\}_{j=1}^{l}$ relatively $\{{{a}_{i}}\}_{i=1}^{l}$ yields a new element of form

\begin{equation} \label{tau}
\left(a_1 {{b}_{{{j}_{1}}}} {{a}_{2}} {{b}_{{{j}_{2}}}}\,...\,\,{{a}_{l}}{{b}_{{{j}_{l}}}} \right)
\end{equation}

At the same time, the squares of any element of form (\ref{tau}) are equal to $\pi$ which was mentioned above. Therefore, a number such elements are exactly $l$.

For instance, the permutation consisting  of  two 3-cycles $s=(123)(456)$ can be presented as square of the following 4 elements: $t_1=(142536)$, $t_2=(152634)$, $t_3=(162435)$, $t_4=(132)(465)$.

In other side by Lemma \ref{odd} each odd cycle has invariant length relatively squaring.
Thus in result of squaring the cycles of odd degree can be obtained as square of cycles of the same degree $l$ or squares of cycles of degree in 2 times grater. Other words in this case we have two cycles structures of root $[l^2]$ or $[(2l)^1]$.
\end{proof}

\begin{cor}\label{NumbRoot}
If ${{g}^{2}}$ contains ${{m}_{l}}$ odd $l$-cycles,  where ${{m}_{l}}=2k+1$ then the number ${{m}_{2l}}$  of $2l$ cycles in $g$ is bounded by $0\le {{m}_{l}}\le k$ the number of  $l$-cycles in $g$ is no less then 1. If  ${{g}^{2}}$ contains  ${{m}_{l}}$ even $l$-cycles,  where ${{m}_{l}}=2n$ and cycles are coprime, then cyclic structure of $g$ has exactly $n$ cycles of degree $2{{l}}$.
\end{cor}
\begin{proof}
The proof immediate from Proposition \ref{StrofRoot} and Lemma \ref{odd}.
\end{proof}

\begin{exm}
  If ${{g}^{2}}$ has cyclic structure $[{{3}^{5}}{{,7}^{3}}]$ then  ${{m}_{3}}$ in  $g$ is bounded by: $1\le {{m}_{3}}\le 5$, ~${{m}_{6}}:$  $1\le {{m}_{6}}\le 2$ and ${{m}_{7}}:$  $1\le {{m}_{7}}\le 3$, ${{m}_{14}}:$   $0\le {{m}_{7}}\le 1$.
\end{exm}


\begin{thm}\label{assymCycl}
An arbitrary element $g$ of ${{A}_{n}}$ having cyclic structure \\ $[({2k})^{m_{2k}},{{(2r-2k)}}^{m_{2r-2k}}]$ can be presented in the form of a product of 2 squares of elements of ${{A}_{n}}$.
\end{thm}


\begin{proof}
First of all let us prove the theorem for the case $m_{2k}=m_{2l}=1$, precisely, $g$ consists of only two different even cycles of degrees $2k$ and $2r-2k$, $k, r \in \textbf{N}$.

Analogically we will be able by the same way to construct the arbitrary number of pairs different even cycles by generating them pair after pair. Also we can generate using $S(A_n)$ arbitrary number of odd cycles since the length of them is invariant while squaring according to Lemma \ref{odd}.



Consider an element $g$ having asymmetric cycle type structure consisting of two mutually coprime cycles of even length. Show that $g$ can be constructed in form of product  $g={{g}_{\text{1}}}\cdot {{g}_{\mathbf{2}}}$ with using  2 joint elements in cyclic representation of  ${{g}_{\text{1}}},\,\,\,{{g}_{\mathbf{2}}} \in A_n$.

\begin{align*}
g = (1, 2, \ldots, 2k-1, 2k) (2k+1, 2k+2, \ldots, 2r-1, 2r),
\end{align*}
where $k\geq 1$ and $r>k$. Then we choose the following permutations $g_1 g_2 \in S(A_n)$:
\begin{align*}
g_1 = (1, 2, \ldots, 2k-1, 2k, 2k+1),\\
g_2 = (1, 2k+2, 2k+3, \ldots, 2r-1, 2r, 2k+1).
\end{align*}
Therefore, we have
\begin{align*}
g_1\cdot g_2 = & (1, 2, \ldots, 2k-1, 2k, 2k+1) \cdot \\
 &(1, 2k+2, 2k+3, \ldots, 2r-1, 2r, 2k+1)=\\ &= (1, 2, \ldots, 2k-1, 2k) (2k+1, 2k+2, \ldots, 2r-1, 2r)=g.
\end{align*}
Note that both permutations $g_1, \, g_2 \in S(A_n)$ because they are presented as $2k + 1$-cycle and $2r-1$-cycle.
If we set $r=2k$ then this case realizes the symmetric cyclic structure $[2k,2k]$.

 This cyclic construction can be generalized.
Let $k$ be even number ($k=2t$, $r>k$). Then the following elements $g_1$, $g_2$ realize factorization of $g$ in even permutations that are presented by odd cycles.
\begin{align*}
  & {{g}_{\text{1}}}=\left( 1,\,\,\,\,\,\,\,\,\,\,...\,\,\,\,\,\,\,\,,k,\,2r \right), \,\, k=2t, \\
 & {{g}_{\mathbf{2}}}=\left( 1,\,\,2k+1,\,\,2k+2,\,\,...\,\,,\,\,2r-1,\,\,2r,\,k+1,\,\,...\,\,,\,2k\, \right), \\
&{g}_{\text{1}}{{g}_{\mathbf{2}}}=\left( 1,2,\,...,\,\,k, \, k+1,\,\,...\,\,,\,2k\, \right) \left(2k+1,\, 2k+2,\,\,2k+3,\,\,...\,\,,\,\,2r-1,\,\,2r\right).
 \end{align*}

Note that both permutations $g_1$ and $g_2$ are even because $k=2t$ therefore they have form of $k+1$-cycle and $2k-k +(2r-2k)+1=2r-k+1$ cycle. Now we mark this permutations $g_1$ and $g_2$ as $g_{11}$ and $g_{12}$. Now we denote $g_1$ by $g_{11}$, and $g_2$ by $g_{21}$.

Changing position of an transition element $2r$ we can construct new pairs of ${{g}_{1}},\,{{g}_{2}}$ which are squares of elements from ${{A}_{n}}$. Since second $g_2$ is $2r - k + 1$ cycle, where first two positions are reserved for elements of second even cycle in $g$.

Now we consider case $m_{2k} = m_{2l}>1$ assuming that first pair is already constructed above we choose the second pair of $2k$ and $2r$ cycles in $g$, which coprime with $g_1g_2$.

According to established above method an element consists of two
even cycles of different degrees can be generated by two odd cycles with joint elements. Consider the second pair of different
even cycles. This pair has empty mutual intersection with the firs pair.


The second pair of $2k$ and $2r$ cycles from cyclic decomposition $g$, permuting new elements can be constructed by utilizing multipliers $g_{21}$, $g_{22}$ by described above method.
 We construct them by the same way as first pair $g_{11}, \,g_{12}$ but on new set of elements having an empty intersection with the first set on which the already constructed substitutions $g_{11}, g_{12}$ act. Using induction on product of cycles pairs with structure $[(2k)^m,(2r)^m]$ we can construct permutation with cyclic structure $[(2k)^{m+1},(2r)^{m+1}]$.



\textbf{Another method of constructing }such permutations ${{g}_{1}},\,\,{{g}_{2}}\in S\left( {{A}_{n}} \right)$
gives the product with  cyclic structure
$[2k,\,\,2l]$, where $k\ne l$ are two pairs of permutations ($g_1$ and $g_2$) having even ranks.
 The first multiplier $g_1$ has rank $2k$ and the second has $2l$, their product has rank $2k+2l-2$.
We obtain such permutations ${{g}_{1}}$ and ${{g}_{2}}$  having  interlinked cycles which intersecting by an odd number of adjacent transpositions in permutation multiplication notation of their representation as a product of transpositions.
Thus in general form we have two elements consisting of even number of transpositions with one joint transposition (or with $2k+1$ joint transpositions). In our construction this transposition $\left( {{\alpha }_{2k}},\,{{\alpha }_{2k+1}} \right)$ is last in the  presentation of ${{g}_{1}}$ and first in the presentation of ${{g}_{2}}$.

\begin{align*}
  & {{g}_{1}}=\left( {{\alpha }_{1}},\,{{\alpha }_{2}} \right)\left( {{\alpha }_{2}},\,{{\alpha }_{3}} \right)...\left( {{\alpha }_{2k}},\,{{\alpha }_{2k+1}} \right) \\
 & {{g}_{2}}=\left( {{\alpha }_{2k}},\,{{\alpha }_{2k+1}} \right)\left( {{\alpha }_{2k+1}},\,{{\alpha }_{2k+2}} \right)\,\,...\,\left( {{\alpha }_{2k+2l-1}},\,{{\alpha }_{2k+2l}} \right),
\end{align*}
\begin{align*}
  & \pi =\left( {{\alpha }_{1}},\,{{\alpha }_{2}} \right)\left( {{\alpha }_{2}},\,{{\alpha }_{3}} \right)...\left( {{\alpha }_{2k-1}},\,{{\alpha }_{2k}} \right)\left( {{\alpha }_{2k+1}},\,{{\alpha }_{2k+2}} \right)\,\,...\,\left( {{\alpha }_{2k+2l-1}},\,{{\alpha }_{2k+2l}} \right)= \\
 & =\left( \,{{\alpha }_{2k}},{{\alpha }_{2k-1}},,...,{{\alpha }_{3}},{{\alpha }_{2}},{{\alpha }_{1}} \right)\left( \,{{\alpha }_{2k+2l}},{{\alpha }_{2k+2l-1}},\,\,...\,\,,{{\alpha }_{2k+2}},{{\alpha }_{2k+1}} \right).
\end{align*}

Note that either condition a) from item 1 or condition 2) of Theorem \ref{Criterion} for these permutations will be satisfied depending on the value of $k$. Therefore ${{g}_{1}},\,\,{{g}_{2}}\in S\left( {{A}_{n}} \right)$ by Theorem \ref{Criterion}. If one imposes additional conditions $2k\equiv 0\left( \bmod 4 \right)$ and $2t\equiv 0\left( \bmod 4 \right)$ then ${{g}_{1}},\,\,{{g}_{2}}\in S\left( {{A}_{n}} \right)$ by item 2.
 Thus, permutation $\pi $ having cyclic structure $[2k,\,2l]$ are presented in form of product of 2 squares.

To show the that arbitrary $2k$-cycle and $2l$-cycle as permutation can be constructed in such a way we recall definition of \textbf{permutation decrement} and recall a result about \textbf{rank of cycle}.

Let $k(\pi )$ be the number of cycles in decomposition of $\pi $.
We call the number $n-k(\pi )$ the \textbf{permutation decrement} $\pi $ and denote this  by $dec(\pi )$.

As well known \cite{Sachk} that the rank of $\pi $  is equal to $dec(\pi )$. Therefore the rank of $n$-cycle is $n-1$. Thus, using $2k-1$ non-coprime and non-equal transpositions of form: $\left( {{\alpha }_{i}},\,{{\alpha }_{i+1}} \right)\left( {{\alpha }_{i+1}},\,{{\alpha }_{i+2}} \right) ... $ we express arbitrary $2k$-cycle.

The alternating proof is completed.

Consider an examples. For instance the elements of even ranks 2 and 6.
${{h}_{1}}=\left( 6,8 \right)\left( 8,7 \right)$,  ${{h}_{2}}=\left( 7,8 \right)  \left( 8, 9 \right)\left( 9,10 \right)\left( 10,11 \right)\left( 11,12 \right)\left( 12,13 \right)\in {{A}_{13}}$ then their product is square:

\begin{align*}
  & g={{h}_{1}}{{h}_{2}}=\left( 6,8 \right)\left( 8,9 \right)\left( 9,10 \right)\left( 10,11 \right)\left( 11,12 \right)\left( 12,13 \right)= \\
 & =\left( 13,12,11,10,9,8,6 \right)\in S({{A}_{13}}). \\
\end{align*}

Note that ${{h}_{2}}=\left( 11,10,9,8,7,12,13 \right)\in S({{A}_{13}})$, ${{h}_{1}},\,\,{{h}_{2}}\in S({{A}_{13}})$, thus constructed in a such way elements are additionally closed by multiplication.

Let us consider an example of constructing of asymmetric cycle type $[2, 6]$ permutation that is in $S({{A}_{13}})$
${{h}_{1}}=\left( 6,7 \right)\left( 8,7 \right)=\left( 6,8,7 \right)$, \\ ${{h}_{2}}=\left( 7,8 \right)\left( 8,9 \right)\left( 9,10 \right)\left( 10,11 \right)\left( 11,12 \right)\left( 12,13 \right)\in {{A}_{13}}$ then their product is square

\begin{align*}
  & g={{h}_{1}}{{h}_{2}}=\left( 6,7 \right)\left( 8,9 \right)\left( 9,10 \right)\left( 10,11 \right)\left( 11,12 \right)\left( 12,13 \right)= \\
 & =(6,7)\left( 13,12,11,10,9,8 \right)\in S({{A}_{13}}). \\
\end{align*}

Note that ${{h}_{2}}=\left( 11,10,9,8,7,12,13 \right)\in S({{A}_{13}})$, ${{h}_{1}},\,\,{{h}_{2}}\in S({{A}_{13}})$, thus constructed in a such way elements are additionally closed by multiplication.

Let us formulate this condition of expressing by 2 squares in terms of generating set of $A_n$.

We recall Mitsuhashi's generating set $M$ is the set of form:  $(1,2)( i,i+1)$, where  $i>1$.
The first cycle having rank $2k$ we have presented in form of product $2k$ transpositions. We can express product of  2 transpositions  by 2 elements  conjugated to elements  of  Mitsuhashi's generating set.
Therefore we need to investigate  structure of  arbitrary such a pair in terms of generators from $M$. Let we need to express  
 the pair $\left( j,\,\,t \right)\left( t,s \right)$ that is in presentation of asymmetric cycles.
To remove symbol $j+1$ by arbitrary $t \in N$: $j<t\le n$ we make $t-j-1$ conjugations of form:
For this we goal consider the pair $\left( 1,2 \right)\left( j,j+1 \right)\in M$. Conjugations by generator $(1,2)(j+1,j+2)$  from $M$ lead us to the following  element on 1 closer to the required element
$\left( 1,2 \right)\left( j+1,j+2 \right)\left( 1,2 \right)\left( j,j+1 \right)\left( 1,2 \right)\left( j+1,\ j+2 \right)=\left( 1,2 \right)\left( j,j+2 \right)$. Finally after $n_{t,j+1}= t-j-1$ conjugations we obtain required element $\left( 1,2 \right)\left( j,\,\,t \right)$.

To generate $2l+1$-cycle it is enough to use $2l$ transpositions. But for each of this transitions of form $(j, t)$ we need to make $n_{t,j+1}= \mid j-t \mid$ conjugations. Thus rank of  $2l+1$-cycle is $2l$. For expressing these transpositions we have to use $ \sum_{i=1} ^{2l}  2n_{t,j+1} + 2l$ generators form $M$.

Thus, for expressing these 2 transpositions  from two generators from $M$ we have used an even number $(2+2(t-j-1))$  of generators from $M$. The form of word is conjugation by of $\left( 1,2 \right)\left( j,j+1 \right)$ by $(2+2(t-j-1))$ generators.

In exactly the same way we can get the second pair of transpositions that is $\left( 1,2 \right)\left( t,\,\,s \right)$. Further, multiplying this  elements  $\left( 1,2 \right)\left( j,\,\,\,t \right)$ and  $\left( 1,2 \right)\left( t,\,\,s \right)$ we have $\left( j,\,\,t \right)\left( t,s \right)=\left( j,\, s, \,\, t \right)$.


\textbf{Remark}. If the necessary condition of Theorem 6 (about ${{m}_{2k}}\equiv 0\left( \bmod 2 \right)$) is satisfied  then condition 2 of Theorem 6 can be formulated it the next form: If there is presentation of $\pi $ in form of unreducible word in alphabet M, with a partition into four disjoint subsets (in general case $4t$ disjoint subsets) of consecutive transpositions in each of which the sum of the degrees of the generators is $2k$-1  i.e.  rank$(C_{2k})$ = $2k$-1 and as a result the degree of the generators is a multiple of 4 in whole $\pi $, then the permutation $\pi $ is  square in ${{A}_{n}}$.

Thus we can reformulate condition 2) of Theorem 6 \textbf{in terms of Mitsuhashi's generating set} $M$. If an element $g\in A_n$ is a product of transposition from $M$ satisfying formulated \textbf{above Remark}, then we shall see that $g$ is a \textbf{square in} $A_n$.

Idea of proof in the following.
 As well known a number of conjugating transpositions is equal to a number of inversions in the cycle $\left( {{i}_{1}},{{i}_{2}},....,{{i}_{k}} \right)$ relative to the base ordered cycle $\left( 1,2,...,k \right)$.

According to Kelly's theorem, each  $n$-cycle can be obtained in ${{n}^{n-2}}$ ways through the product of $n-1$ transposition. One of the well-known ways of representing the cycle $\left( 1,2,...,k \right)$ is a product of form \[\left( 1,k \right)\left( k,k-1 \right)\left( k-2,k-\,3 \right)...\left( 2,3 \right)=(1,\,\,...\,\,,k)\].

Thus, the rank of unreducible product of element satisfying condition 2 of Theorem 6 is $4(2k-1)$.
\end{proof}


This Theorem immediately entails the following more general Theorem.

\begin{thm}\label{Prod}
 {\sl The verbal width of $V_{S(A_n)}(A_n) = A_n$ is 2, where $n>3$.} 
\end{thm}
According to proved above statements the cycle structure of elements $g\in A_n$ constructed by us
 can be divided into three types of cycles. The first type consists of pairs of cycles of the same even degree  mentioned in Lemma \ref{2l}.
This way is based on squaring of $2l$-cycles described in Lemma \ref{2l}. Also it can be reached by using Theorem \ref{assymCycl}.
The second type is pairs of cycles of the different even degrees arising from product of non mutually coprime cycles which described in Theorem \ref{assymCycl}. Also the third type is arbitrary number of cycles of odd degree that were established in Lemma \ref{odd} and Lemma \ref{2l} about square of $2l$ cycles, when $l \equiv 1 (mod 2)$.


According to Theorem \ref{assymCycl}, Proposition \ref{necessaryCondofSquare} and Lemmas \ref{odd}, \ref{2l} there are no other types of cycle structures of element.

  On the base of just mentioned structure we can established the following.
The element with $[(2k_0)^{l_{2k_0}}, (2m_0)^{l_{2k_0}}], \ldots, [(2k_j)^{l_{2k_j}}, (2m_j)^{l_{2k_j}}]$ can be constructed due to successively applying the procedure described in the proof of Theorem \ref{assymCycl}.

We denote by ${{C}_{ij}}\left( \pi  \right)$ a pair of cycles of length $i$ of type $j$ from presentation of permutation $\pi $, $1\le j\le 2$, for such $j$ a lenght $i$ is even. If  $j=3$ then ${{C}_{ij}}\left( \pi  \right)$ is one odd cycle from $\pi \in {{A}_{n}}$. So $i$ is odd for $j=3$.
A number  of  ${{C}_{ij}}\left( \pi  \right)$ in cyclic decomposition of $\pi $ is denoted by  ${{v}_{ij}}\left( \pi  \right)$.
Let $\pi \in {{A}_{n}}$ consists of ${k_{1,2}}$ cyclic multipliers of the types 1 and 2 and $k_3$ multipliers of third type.  Such cyclic decomposition  of  $\pi \in {{A}_{n}}$ satisfies to the equality
 $$\sum\limits_{i=1, \, 1<j<3}^{k_{1,2}}{2i{{v}_{ij}}\left( \pi  \right)}+ \sum\limits_{i=1}^{k_3}{i{{v}_{i3}}\left( \pi  \right)} =n.$$


We can construct an arbitrary many cycles of odd degree utilizing Lemma \ref
{odd} or Lemma \ref{2l} by induction on different support for each cycle.

In an arbitrary $ \pi \in A_n$ there are no other cycles besides of these three types
because even permutation $ \pi$ can contain only even number of even cycles \cite{Seg}. Therefore a set of cycles with even degree consist of pairs with cyclic structure $[2k, 2m]$, where $k,m \in  \mathbb{N}$ (asymmetric pairs) and pairs of even cycles with the same degrees (symmetric pairs).

These symmetric and asymmetric pairs are completely exhausted by the aforementioned $C_{i1}$ and $C_{i2}$.
  All cycles of odd degree contains in type 3 ($C_{i3}$) and can be constructed by Lemma \ref{odd} or Lemma \ref{2l} as elements with verbal width 1.

The equality $V_S(A_n) = A_n$ was established in Property \ref{closedness}.
Thus, an arbitrary element of $A_n$ can be presented in form of a product of 2 squares of elements from $A_n$.

\begin{cor}
An arbitrary element $g$ of ${{A}_{n}}$ having cyclic structure $[({2k})^{},({2r})^{}]$ can be presented in form of a product of 2 squares by $2k+2r-2$ ways.
\end{cor}


\begin{proof}
We count the number of different permutations of the type ${{g}_{1}}$ and ${{g}_{2}}$ consisting from product of mutually coprime cycles of length $2k$ and $2l$. Let ${{\beta }_{1}}$ is on first position because of its position is not decisive for sequence of elements in cycle. 
We can variate position of transition element $t$ relatively to the element ${{\beta }_{1}}$ position changing the sequence in the original cycles presenting $g_1$ and $g_2$. A number of such position is $2k+2r-2$. We set $n,m \equiv 0 (mod 2)$.
At the same time, the sequence in the original cycle of the supplied elements
\begin{align*}
  & {{g}_{\text{1}}}=\left( {{\beta }_{1}},\,\,\,\,\,\,\,...\,,\,\,\,\,\,{{\beta }_{m}},\,\,\,\,\,\,\,t,\,\,\,\,{{\alpha   }_{1}},\,{{\alpha }_{2}},\,\,...,\,\,{{\alpha }_{n}}\,\, \right) \\
 & {{g}_{\mathbf{2}}}=\left( {{\beta }_{1}},\,{{\alpha }_{n+1}},\,\,\,....\,,\,\,\,\,{{\alpha  }_{2k-1}},\,\,\,\,\,\,t,\,\,{{\beta }_{m+1}},\,.\,...\,,\,\,{{\beta }_{2r-1}},\,\,{{\beta }_{2r}} \right) \\
\end{align*}

Note $g_1, g_2$ are cycles of odd degree so $g_1, g_2 \in S(A_n)$ by Theorem \ref{Criterion} and Lemma \ref{odd}.

\begin{multline*}
\quad{{g}_{1}}{{g}_{\mathbf{2}}}=\left( {{\beta }_{1}},\,...\,\,,\,\,\,\,\,{{\beta }_{m}},\,\,\,{{\beta }_{m+1}},\,.\,...\,,\,\,{{\beta }_{2r-1}},\,\,{{\beta }_{2r}} \right) \\
\left( \,t,\,\,{{\alpha }_{1}},\,{{\alpha }_{2}},\,\,...,\,\,{{\alpha }_{n}},{{\alpha }_{n+1}},\,\,\,....\,,\,\,\,\,{{\alpha }_{2k-1}} \right).
\end{multline*}
This  completes the proof.
\end{proof}



\section{The verbal width of Mathieu groups}
Taking in account some exceptional isomorphism $PSL_2(F_9) \simeq A_6$ as well as $PSL_2(F_4)\simeq A_5$ \cite{Doc} we obtain the verbal width by set $S(PSL_2(F_9))$ of $PSL_2(F_9)$ is 2. This yields the \textbf{assumption} that verbal width by set $S(PSL_2(F_{p^n}))$ of $PSL_2(F_{p^n})$ is 2.

We denote by $H_{j}$ the subgroup generated by $S(M_{j})$.

\begin{prop}  The following Mathieu groups are not generated by their sets of squares: $M_8, \, M_9, \, M_{10}$. Other Mathieu groups are generated of by their sets of squares. The diameter of  $M_8, \, M_9, \, M_{10}$ over $\textbf{S} (M_8), \, \textbf{S} (M_9), \, \textbf{S} (M_{10})$ is $\infty$.
\end{prop}

\begin{proof}
Firstly we recall exceptional isomorphism ${{M}_{9}}\simeq PS{{U}_{3}}(2)$ \cite{Doc}.
The fact of the non equalities: $S(M_8)\neq M_8,  S(M_9)\neq M_9, S(M_{10}) \neq M_{10}$ was established by us due to computation on  CAP system.
  All other prime Mathieu groups $M_{11}, \, M_{12}, \, M_{21}, \, M_{23}, \, M_{24}$ are generated by squares.

Because of if we assume $H_{11}$, $H_{12}$,  $H_{21}$, $H_{23}$, $H_{24}$  are proper subgroups in the correspondent group $M_{11}, M_{12},  M_{21},  M_{23}, M_{24}$ then it will be proper normal subgroup which is contradiction.
  Because of in oppositive case the subgroup ${{G}^{2}}$ is proper normal subgroup in $M_{12}$. And the same is true for other prime Mathieu groups.
Indeed ${{M}_{12}}\simeq PS{{L}_{2}}({{F}_{11}})$ that's why $M_{12}$ is prime as $PS{{L}_{2}}({{F}_{11}})$. As well known ${{M}_{12}}\simeq PS{{L}_{2}}({{F}_{11}})$ and  ${{M}_{21}}\simeq PSL_3(F_4)$, gives us the same way to result $vw(({M}_{12}, S({M}_{12})))=2$.
\end{proof}

\begin{prop}
The verbal width of $V_{S(M_i)}(M_i)$ for $8 \leq i \leq 10$ equals to 1.
\end{prop}
The structure of correspondent group $M_9= (C_3 \times C_3) \rtimes Q_8$ and subgroup $H_9= (C_3 \times C_3) \rtimes C_2$ give us way to find and understand a value of verbal width of $H_9$. Indeed $(C_3 \times C_3) \rtimes C_2$ has two element minimal generating set which can be presented as squares. Thus verbal width at least is 2. Since every commutator can be  generated no more than three squares of group elements then it can not be grater then 3. Further, $M_{10}= A_6*C_2$.

\begin{thm} The verbal width of $V_{S(M_i)}(M_i)$ in $M_{11}, \,  M_{12}$, $M_{20}$, till $M_{24}$ is equal to 2.
\end{thm}
\begin{proof} Considering
$M_{11} $, we denote correspondent verbal subgroup $V_{S(M_{11})}(M_{11})$ as $H_{11}$.
The computation provided out using the GAP system gave us $M_{20} \simeq  (C_2 \times C_2 \times C_2 \times C_2) : A_5$ that means $M_{20} \simeq  (C_2 \times C_2 \times C_2 \times C_2) \rtimes A_5$.
The following equality $H_{11}=M_{11}$, obtained due to GAP system, means that set $S(M_{11})$ generates whole $M_{11}$.  
\end{proof}

Due to exceptional isomorphism
 ${{M}_{12}}\simeq PS{{L}_{2}}({{F}_{11}})$ we can apply the mentioned above criterion of squreness to $M_{12}$.

\section {Square root of an element in $GL_2(\mathbb{F}_p)$, $SL_2(\mathbb{F}_p)$ and $PSL_2(\mathbb{F}_p)$}

Let $SL_2(F_p)$  $(PSL_2(F_p))$ denote the special (projective) linear group of degree 2 over a finite field of order $p$. And a degree always means an irreducible character degree in this paper.

We recall the well known relation between eigenvalues of $A$ and $f(A)$.
\begin{lem}\label{eigenval}
(Lemma 2) if $\beta $ is an eigenvalue for $B$ then ${{\beta }^{2}}$ is an eigenvalue for ${{B}^{2}}$.
\end{lem}

Consider the criterion squareness of elements of $PSL_2(\mathbb{F}_p)$ as well as of $SL_2(\mathbb{F}_p)$ which can be presented by diagonal matrix,  where $\mathbb{F}_p$ is some finite field. As well known \cite{LinearShpringer} a matrix can be presented in the diagonal form iff algebraic multiplicity of its eigenvalues are the same  as a geometric multiplicity.

\begin{thm}\label{CriterionPSL}
Let
$A \in PSL_2(F_p)$  and eigenvalues of $A$ has algebraic multiplicity equal to geometric multiplicity \cite{Linear, LinearShpringer}, then
for matrix $A \in PSL_2(\mathbb{F}_p)$, there is a solution  $B \in PSL_2(\mathbb{F}_p)$ of matrix equation
 \begin{equation*}
X^2 = A
\end{equation*}
if and only if, at least one of

 \begin{equation}\label{tr}
 tr(A)+2,  \, \, tr(A)-2	
\end{equation}

  is a quadratic residue or $0$ in $\mathbb{F}_p$.

In case  $A\in S{{L}_{2}}({{F}_{p}})$  then last condition \eqref{tr} takes form:
\[\left(\frac{trA+2}{p}\right) \in \{0,1\}. \]

In case $A \in GL_2 ({{F}_{p}})$ then this condition has form
 \begin{equation*}
\left(\frac{trA + 2\sqrt{detA}}{p} \right) \in \{0,1\}.
 \end{equation*}

\end{thm}

\begin{proof}
We assume that the matrixes $A$ and $B$ eigenvalues $\lambda_1, \lambda_2 $ and $\mu_1, \mu_2 $ respectively.
Let characteristic polynomial ${{\chi }_{B}}(x)$ of $B$ is the following: ${{\chi }_{B}}(x) = (x - \mu_1)(x - \mu_2)$. We denote $tr(A)$ by $a$. Then $a = \mu_1^2 + \mu_2^2$ is the sum of the roots of a polynomial ${{\chi }_{A}}(x)=(x-\mu_1^2)(x-\mu_2^2)$. Then $a = \mu_1^2 + \mu_2^2 = (\mu_1 + \mu_2)^2 - 2\mu_1\mu_2$. Since determinant of $PSL_n(F_p)$ is  1 then the product of eigenvalues of matrix $A$ satisfy the following equality: $\mu_1^2\mu_2^2 = 1 $ imlies $\mu_1\mu_2 = \pm 1$. Then $a + 2\mu_1\mu_2 = a \pm2 = (\mu_1+\mu_2)^2$. As well known a trace and determinant of matrix $A$ belongs to $F_p$. Therefore $\mu_1 + \mu_2 \in F_p$ and $\mu_1 \mu_2 \in F_p$.

We show existing of  ${{\chi }_{B}}(x):={{x}^{2}}-cx+1$ who's roots ${{\mu }_{1}},\,\,{{\mu }_{2}}$  are e.v. of $B$.
We prove the existence of ${{\chi }_{B}}(x):={{x}^{2}}-cx+1$ who's roots are e.v. ${{\mu }_{1}},\,\,{{\mu }_{2}}$  of$B$.  Let ${{\chi }_{{{B}^{2}}}}(x)={{\mu }^{2}}-a\mu +1$ then $\mu _{1}^{2},\,\,\,\mu _{2}^{2}$ are e.v. for $A$ and according to Viet theorem $\mu _{1}^{2}+\mu _{2}^{2}=a$.
Let us prove that the condition $(\frac{trA+2}{p})=1$ and the fact, following from Lemma \ref{eigenval}, that $\mu _{1}^{2},\,\,\,\mu _{2}^{2}$ are e.v. of  A  yields that there are  ${{\chi }_{B}}(x)$ over ${{\text{F}}_{\text{p}}}$  with roots  ${{\mu }_{1}},\,\,\,{{\mu }_{2}}$.

By condition of the theorem $trA+2$ is quadratic residue or 0, then there is $\sqrt{tr(A)+2}=\sqrt{{{({{\mu }_{1}}+{{\mu }_{2}})}^{2}}}$ in ${{\text{F}}_{\text{p}}}$, whence we have $tr(\text{B})\in {{\text{F}}_{\text{p}}},\,\,\det B\in {{\text{F}}_{\text{p}}}$ therefore ${{\chi }_{B}}(x)$ has coefficients in ${{\text{F}}_{\text{p}}}$, hence $B$ presented in Frobenius form belongs to $S{{L}_{2}}({{\text{F}}_{\text{p}}})$. Furthermore matrix $B$ having  e.v.${{\mu }_{1}},\,\,\,{{\mu }_{2}}$ is matrix over ${{\text{F}}_{\text{p}}}$, but ${{\mu }_{1}},\,\,\,{{\mu }_{2}}$ can be from ${{\text{F}}_{{{\text{p}}^{2}}}}\backslash {{\text{F}}_{\text{p}}}$.

 In the same time $\mu _{1}^{2}+\mu _{2}^{2}={{\left( \mu _{1}^{{}}+\mu _{2}^{{}} \right)}^{2}}-2=tr(A)$.
	Consider characteristic polynomial ${{x}^{2}}-cx+1=0$, where $c={{\mu }_{1}}+{{\mu }_{2}}$ . Show that exactly it is characteristic polynomial for $B=\sqrt{A}$. Indeed, since e.v. of B have to be ${{\mu }_{1}},\,\,\,{{\mu }_{2}}$ and have to satisfy condition (1).  Substitute ${{\mu }_{1}},\,\,\,{{\mu }_{2}}$ in ${{x}^{2}}-cx+1=0$ we can verify that its are roots ${{x}^{2}}-({{\mu }_{1}}+{{\mu }_{2}})x+1={{\mu }_{1}}^{2}-({{\mu }_{1}}+{{\mu }_{2}}){{\mu }_{1}}+1=0-{{\mu }_{1}}{{\mu }_{2}}+1=0$ and the analogously for verification for ${{\mu }_{2}}$.

Using condition \eqref{tr}
 we show that matrix $B$ with trace $tr(\text{B})={{\mu }_{1}}+{{\mu }_{2}}$ and determinant ${{\mu }_{1}}{{\mu }_{2}}=1$ exists. For this goal we express ${{\left( \mu _{1}^{{}}+\mu _{2}^{{}} \right)}^{2}}=\mu _{1}^{2}+\mu _{2}^{2}+2=tr(A)+2$.  Since $(\frac{trA+2}{p})=1$ then square root $\sqrt{tr(A)+2}=\sqrt{{{({{\mu }_{1}}+{{\mu }_{2}})}^{2}}}$ exists, i.e. matrix $A$ having trace ${{\mu }_{1}}+{{\mu }_{2}}$ exists.

The matrix $B$ is an element from $PSL_2(\mathbb{F}_p)$, it yields $\mu_1\mu_2 = 1$. Therefore, the
matrices with $\mu_1\mu_2 = - 1$ are not valid.

\[B=\left( \begin{array}{cc}
   \mu_1 \,\,\,& 0 \\
 0\,\,\,  & \mu_2  \\
\end{array} \right).\]

 The case, where $ \mu_1 = \mu_2= \mu$ implies that $\mu=g^2$, where $g\in F_p$ this condition demands to $\mu$ be a square in $F_p$. In this case $\mu \in F_p$ where $F_p$ is base field, so according to the theorem about Jordan basis Jordan matrix exists over this field. Considering the multiple eigenvalues $ \mu_1 = \mu_2= \mu$ and Viet theorem for characteristic polynomial, where that $b\in F_p$ holds, we see $b= \mu_1 + \mu_2 = 2\mu$. So transformation of matrix $A$ in Jordan form is possible. Note that in case $p=2$ expression $ \mu = b/2 $ is impossible. But existing of such $ \mu $ can be proven in another way.

There are two forms of matrix possible,
either $A$ has 2 multiple eigenvalues or $A$ has different eigenvalues.  If $A$ has form of a Jordan block then this case satisfies the conditions of the theorem \ref{CriterionPSLJordan} as well when  $A$ is scalar matrix then it satisfy condition of the theorem described below.






We note that case when $\lambda=0$ is irrelevant for consideration, because of belonging $A$ to $PSL_2(\mathbb{F}_p)$ demands $det(A)=1$.



 Two different eigenvalues $\lambda_1,  \lambda_2$ occur only in case of diagonal matrix. These eigenvalues: $\lambda_1= \mu_1^2,  \lambda_2=\mu_2^2$. In case of Jordan block of size 2 on 2 only multiple eigenvalues are possible.


A sufficient condition in the subcase of the Jordan cell being obvious is the same existence of a root $\sqrt{\lambda }$ in the finite field ${{F}_{p}}$, $p>2$. It is this root that is the eigenvalue \[\mu =\sqrt{\lambda }\] of the matrix $B$.

Let us express the square by trace and determinant. Since $ \mu_1^2 + \mu_2^2 = (\mu_1 + \mu_2)^2  - 2\mu_1\mu_2$ then $(\mu_1 + \mu_2)^2  = \mu_1^2 + \mu_2^2  + 2\mu_1\mu_2 = tr(A) + 2$. A value in the left part is the quadratic residue so $tr(A) + 2$ is too.
 Since in $PSL_2(F_p)$ and $A$ and $-A$
 are identified according to the factorization by the center of the $SL_2(F_p)$ then in case of $PSL_2(F_p)$ the both signs $"\pm"$ before the $tr(A)$ fit. In case of $SL_2(F_p)$ this conditions transforms in the form $tr(A) \pm2$ due to possible root $B$ with $det(B)=-1$. But such case $B \notin SL_2(F_p)$. If $tr(A)+2$ is multiple to $p$ then characteristic polynomial of $A$ ($\chi_A(x)$) takes form $x^2+1$. In general case $\chi_A(x)$ has form  $x^2 -(tr(A))x+detA$.

There is a similar criterion for a polynomial. In particular in role of such polynomial could play characteristic polynomial of $A$ and $B$.
Let be ${{\chi }_{A}}(x)(x) = x^2 - ax + 1$ $(a \in F_p)$. We call this polynomial squared if there is such ${{\chi }_{B}}(x) = x^2 - cx + d$, that $d=\mu_1 \mu_2, \, c= \mu_1 + \mu_2, \,  c,d \in F_p. $

The eigen values  $\mu_1$, $\mu_2$ are determined as $\mu_i = \pm \surd{(\mu_i)^2}, \,\, i \in \{1,2\}$. Furthermore $\mu_i= \frac{ x \pm \sqrt{y}} {2} $.
A combination of signs in terms representing $\mu_1$ and $\mu_2$ have to satisfy condition of polynomial squared: $d=\mu_1 \mu_2, \, c= \mu_1 + \mu_2, \in F_p$.

To prove if $A\in S{{L}_{2}}({{F}_{p}})$  then condition
\[(\frac{trA+2}{p})\in \{0,1\}\]
is sufficient for the existence of a solution of the equation $A={{X}^{2}}$
we shall prove that $\exists B:{{B}^{2}}=A$. Let us prove that this condition $(\frac{trA+2}{p})\in \{0,1\}$  provides the existing of characteristic polynomial  $\chi_B(x)={{x}^{2}}-cx+1$  with the $c\in {{F}_{p}}$ whose roots are the eigenvalues $\mu _{1}^{{}},\,\,\mu _{2}^{{}}$  of the matrix $B$ and ${{\mu }_{1}}{{\mu }_{2}}=1$ provided that $\mu _{1}^{2},\mu _{2}^{2}$  are the eigenvalues of matrix  $A$.

The proof is very similar in case $A\in PS{{L}_{2}}({{F}_{p}})$ ( $GL_2 ({{F}_{p}})$ ) then the condition $(\frac{trA \pm 2}{p})=1$  $((\frac{trA + 2\sqrt{detA}}{p})=1)$ is enough  for solvability of the equation  $A={{X}^{2}}$.

Consider the e. v.,  denoted by $\lambda_1$ and $\lambda_2$,  of matrix $A$.  Then we need to find such $\mu _{1}^{{}},\,\,\mu _{2}^{{}}$ that ${{\lambda }_{1}}=\mu _{1}^{2},\,\,{{\lambda }_{2}}=\mu _{2}^{2}$ and ${{\mu }_{1}}{{\mu }_{2}}=1$.

Note that existing of such $\mu _{1}^{{}},\,\,\mu _{2}^{{}}$  implies the following condition
${{({{\mu }_{1}}+{{\mu }_{2}})}^{2}}=\mu _{1}^{2}+2{{\mu }_{1}}{{\mu }_{2}}+{{\mu }_{2}}+{{\mu }_{2}}=\mu _{1}^{2}+\mu _{2}^{2}+2={{\lambda }_{1}}+{{\lambda }_{2}}+2=trA+2.$
holds.  Consider  such $c$ that ${{c}^{2}}=trA+2,\,\,\,c\in {{F}_{p}}$. We prove that roots $\mu _{1}^{{}},\,\,\mu _{2}^{{}}$ of ${{\mu }^{2}}-c\mu +1$
endowed  the property  ${{\mu }_{1}}{{\mu }_{2}}=1$  and $\mu _{1}^{2},\,\,\mu _{2}^{2}$  coincide with ${{\lambda }_{1}},\,\,{{\lambda }_{2}}$.

Indeed, the equality ${{\mu }_{1}}{{\mu }_{2}}=1$ is provided by Viet theorem. It remains to show that $\mu _{1}^{2}={{\lambda }_{1}},\,\,\mu _{2}^{2}={{\lambda }_{2}}$ this is equivalent to we need to show the equality of polynomials $\chi(A) = \left( x-{{\lambda }_{1}} \right)\left( x-{{\lambda }_{2}} \right)$ and $\left( x-\mu _{1}^{2} \right)\left( x-\mu _{2}^{2} \right)$.
We have
$\left( x-{{\lambda }_{1}} \right)\left( x-{{\lambda }_{2}} \right)={{x}^{2}}-\left( {{\lambda }_{1}}+{{\lambda }_{2}} \right)x+{{\lambda }_{1}}{{\lambda }_{2}}={{x}^{2}}-\left( trA \right)x+\det A={{x}^{2}}-\left( trA \right)x+1$.  In the similar way $\left( x-\mu _{1}^{2} \right)\left( x-\mu _{2}^{2} \right)={{x}^{2}}-\left( \mu _{1}^{2}+\mu _{2}^{2} \right)x+\mu _{1}^{2}\mu _{2}^{2}={{x}^{2}}-\left( \mu _{1}^{2}+\mu _{2}^{2}+2{{\mu }_{1}}{{\mu }_{2}}-2{{\mu }_{1}}{{\mu }_{2}} \right)x+\mu _{1}^{2}\mu _{2}^{2}$.
Since $\mu _{1}^{{}},\,\,\mu _{2}^{{}}$  are roots of  ${{\mu }^{2}}-c\mu +1$ then according to Viet theorem: ${{\mu }_{1}}{{\mu }_{2}}=1$ and  ${{\mu}_{1}}+{{\mu }_{2}}=c$, This implies
$\left( x-\mu _{1}^{2} \right)\left( x-\mu _{2}^{2} \right)={{x}^{2}}-\left( {{c}^{2}}-2 \right)x+{{1}^{2}}$.  Taking into account that ${{c}^{2}}=trA+2$ we derive  $\left( x-\mu _{1}^{2} \right)\left( x-\mu _{2}^{2} \right)={{x}^{2}}-\left( trA+2-2 \right)x+1={{x}^{2}}-\left( trA \right)x+1$,  which coincides with the characteristic polynomial $\chi(A)= \left( x-{{\lambda }_{1}} \right)\left( x-{{\lambda }_{2}} \right)$ for the matrix $A$. This completes the proof.

In cases when $A\in PS{{L}_{2}}({{F}_{p}})$ or $A\in G{{L}_{2}}({{F}_{p}})$ the proof is very similar.
\end{proof}

\begin{cor}
If $A\in GL(F_2)$ the condition \ref{tr} takes form:
\[\left(\frac{trA}{p}\right) \in \{0,1\}. \]
\end{cor}

\begin{cor} If matrix $A$ admits  diagonal form over ${{\text{F}}_{2}}$ then $A$ is square over ${{\text{F}}_{2}}$.
\end{cor}
\begin{proof}
Since in $F_{2^n}$ all elements $g_i \in F_{2^n}$ are quadratic squares, therefore a diagonal matrix $A$ is always square of the mentioned above $B$ over $F_{2^n}$.
\end{proof}

\begin{prop} If matrix $A$ do not admits diagonal form over ${{\text{F}}_{2}}$ then $A$ is not square in $G{{L}_{2}}\left( {{\text{F}}_{2}} \right)$ over ${{\text{F}}_{2}}$.
\end{prop}
\begin{proof}
If matrix $A$ do not admits diagonal form over ${{\text{F}}_{2}}$ then $A$ is not square in $G{{L}_{2}}\left( {{\text{F}}_{2}} \right)$ over ${{\text{F}}_{2}}$.
We consider equation of form ${{X}^{2}}=A$ and show that it has not solutions over ${{\text{F}}_{2}}$ в $S{{L}_{2}}\left( {{\text{F}}_{2}} \right)$ in case ${{\chi }_{A}}\left( x \right)={{\mu }_{A}}\left( x \right)$. The conditions of theorem implies that geometrical dimension  of e.v. is  1 but algebraic multiplicity of e.v. $\lambda $ is 2. We make prof by contradiction, assuming that is true then

${\left(\begin{array}{cc}
  & \lambda \,\,1 \\
 & 0\,\,\lambda  \\
\end{array} \right)}^{2}=\left(\begin{array}{cc}
  & {{\lambda }^{2}}\,\,2\lambda  \\
 & 0\,\,\,\,{{\lambda }^{2}} \\
\end{array} \right)$  but
$\left( \begin{array}{cc}
  & {{\lambda }^{2}}\,\,2\lambda \\
 & 0\,\,\,\,{{\lambda}^{2}} \\
\end{array}\right)=
 \left(\begin{array}{cc}
  & {{\lambda }^{2}}\,\,0 \\
 & 0\,\,\,\,{{\lambda }^{2}}\\
\end{array} \right)$ over ${{\text{F}}_{2}}$.
That contradicts to condition of this Theorem.
\end{proof}

Let $B$ has characteristic polynomial ${{x}^{2}}+bx+c$. It is well known that trace of $B$ is stable
under choosing of vector space base.

We denote Jordan form of matrix $A$ as ${{J}_{A}}$.

\begin{lem}\label{Trace} If a matrix $A\in PS{{L}_{2}}({{F}_{p}})$ has multiple eigenvalues ${{\beta }_{1}}={{\beta }_{2}}=\beta $  and non-trivial Jourdan block of size $2\times 2$ then $\beta \in {{\text{F}}_{p}}.$
\end{lem}

Proof.  Since in this case eigenvalues are presented as elements of matrix $B$  standing on diagonal, then this matrix can be in form: $B=\left( \begin{array}{cc}
   \beta \,\,\,& 1 \\
  0\,\, & \beta  \\
\end{array} \right)$   or

\[B=\left( \begin{array}{cc}
   \beta \,\,\,& 0 \\
 0\,\,\,  & \beta  \\
\end{array} \right).\]

But the eigenvalues  of the matrix are multiples, therefore $\beta +\beta =tr(B)\in {{\text{F}}_{p}}$. This implies $2\beta =b$, therefore in a field of characteristic  non equal 2  we express this eigenvalue as $\beta =\frac{b}{2}$. Hence $\beta \in {{\text{F}}_{p}}$.  The proof is complete.

\begin{thm}\label{CriterionPSLJordan}
 Under conditions $(\frac{\lambda }{p})=1$ in ${{F}_{p}}$ and matrix $A$ is similar to a Jordan block of the form
\begin{equation}\label{Jblock}
 {{J}_{A}}= \left(
  \begin{array}{cc}
    \lambda & 1 \\
     0 & \lambda \\
  \end{array}
\right)
\end{equation}
a square root $B$ of $A$ exists in $PS{{L}_{2}}({{F}_{p}})$.
\end{thm}

\begin{proof}
Assume that square from $A$ exists in $PS{{L}_{2}}({{F}_{p}})$ (or in $S{{L}_{2}}({{F}_{p}})$ correspondently). We denote matrix $B$ transformed to upper triangular form by $U{{T}_{B}}$. Let us show that there that provided condition above it always exists such $B:\,\,UT_{B}^{2}={{J}_{A}}$, where $U{{T}_{B}}$ is $B$ transformed to UTM form.  Then we show that it implies existing  of solution of \[{{X}^{2}}=A. \]
From the existence of the Jordan block for $A$ follows the existence of a similarity transformation $U$ transforming ${{B}^{2}}$ to the Jordan normal form $J_B$ because of  $A = {{B}^{2}}$ and A has non-trivial Jordan block denoted by $J_A$. But square root from $B^2$ this operator $U$ transforms in upper triangular form $UT_B$. Then if we find solution for

\begin{equation}\label{uppertriang}
UT_B^2=J_A
\end{equation}

we can obtain solution for ${{X}^{2}}=A$ because of the following:

\begin{equation}\label{eqtransform}
A=U\cdot (U{{T}_{B}})^2\cdot {{U}^{-1}}=(U\cdot U{{T}_{B}}\cdot {{U}^{-1}})(U\cdot U{{T}_{B}}\cdot {{U}^{-1}})={{B}^{2}}.
\end{equation}

It means that such matrix $U{{T}_{B}}$ satisfying \eqref{eqtransform}, exists and it can be transformed  by the same similarity transformation by conjugation in form $U{{T}_{B}}={{U}^{-1}}BU$ by the same matrix that transforms $A$ in ${{J}_{A}}$ because of ${{B}^{2}}=A$. To show the existing of such solution of \eqref{uppertriang} we acting by invers transformation $A=U\cdot (U{{T}_{B}})^2\cdot {{U}^{-1}}=(U\cdot U{{T}_{B}}\cdot {{U}^{-1}})(U\cdot U{{T}_{B}}\cdot {{U}^{-1}})={{B}^{2}}$, where $U$ is similarity transformation $B$ to

$$U{{T}_{B}}=\left(
  \begin{array}{cc}
  & \beta \,\,\,\gamma  \\
 & 0\,\,\,\beta  \\
  \end{array}
 \right).$$
note that its diagonal elements ${{b}_{11}}={{b}_{22}}=\beta $ are the same. Therefore according to Lemma \ref{Trace} we have $\beta \in {{F}_{p}}$. Even more easier we can deduce it without Lemma 19. We have ${{b}_{11}}={{b}_{22}}=\beta $, then  $\beta +\beta =\text{Tr}({{U}^{-1}}BU)$. Therefore $2\beta \in {{\text{F}}_{p}}$. It implies that $\beta \in {{F}_{p}}$ if $p>2$.

$$(UT_{B})^{2}=\left( \begin{array}{cc}
  {{\beta }^{2}}\,\,\,  &  2\beta \gamma  \\
  0\,\,\,\,\,\,\,\,\,  & {{\beta }^{2}} \\
\end{array} \right). $$

Here the element $\gamma $ can be chosen $\gamma :\,\,\, 2\beta \gamma =1$ so $\gamma = {2 \beta}^{-1}$ taking into account that $\beta =\sqrt{\lambda }$ which is already determined by $A$. Then  $(UT_{B})^{2}:$

$(UT_{B})^{2}=\left( \begin{array}{cc}
   {{\beta }^{2}}\,\,\, & 2\beta \gamma  \\
  0\,\,\,\,\,\,\,\,\, & {{\beta }^{2}} \\
\end{array} \right)=\left( \begin{array}{cc}
   {{\beta }^{2}}\,\,\, & 1 \\
  0\,\,\,\,\, & {{\beta }^{2}} \\
\end{array} \right)=J_A = \left(
	\begin{array}{cc}
	\lambda \, & 1 \\
	0 & \lambda \, \\
	\end{array}
\right) $.

Furthermore we show that these conditions is also necessary but not only sufficient. It means if $(\frac{\lambda }{p})=-1$, then there are no matrix $B$ over $PS{{L}_{2}}({{F}_{p}})$ such that ${{B}^{2}}=A$. By a reversal of theorem condition  and using the representation in the form of UTM for and for we see that $B$ from $PS{{L}_{2}}({{F}_{p}})$ such that ${{B}^{2}}=A$. We see that according to the Lemma \ref{Trace}
 the eigenvalue of $B$ over lie in the main field ---  ${{\text{F}}_{p}}$. However, we assumed that  $(\frac{\lambda }{p})=-1$. Thus we obtain the desirable  contradiction.

Let us show that condition of  non-diagonalizability of matrix is necessary in the conditions of this Theorem. By virtue of the well-known theorem stating that if the algebraic multiplicity is equal to the geometric  multiplicity for each eigenvalue, then matrix is diagonalizable otherwise it is not diagonalizable, we see that if the condition of  similarity  to  ${{J}_{A}}=\left( \begin{array}{cc}
   \lambda \,\,\, & 1 \\
  0\,\,\, & \lambda  \\
\end{array} \right)$
indicated in this Theorem \ref{CriterionPSLJordan}  does not holds, then  such $A$ satisfy the conditions of this Theorem \ref{CriterionPSL}, where algebraic multiplicity is equal to geometrical. And since the condition \ref{Jblock} of this criterion is nature, therefore, it is no longer necessary to prove the non-diagonalizability condition in Theorem \ref{CriterionPSLJordan}.

Proof of \textbf{necessity}.
Furthermore we show that these conditions is also necessary but not only sufficient. It means if $(\frac{\lambda }{p})=-1$, then there are no matrix $B$ having non trivial Jordan block over $PS{{L}_{2}}({{F}_{p}})$ and $S{{L}_{2}}({{F}_{p}})$ such that ${{B}^{2}}=A$. By a reversal of theorem condition  and using the representation in the form of UTM for and for we see that $B$ from $PS{{L}_{2}}({{F}_{p}})$ such that ${{B}^{2}}=A$. We see that according to the Lemma  the eigenvalue of $B\in S{{L}_{2}}({{F}_{p}})$or $B\in PS{{L}_{2}}({{F}_{p}})$ correspondingly,  lie in the main field -- ${{\text{F}}_{p}}$. Furthermore according to Lemma \ref{eigenval} if $\beta $ is an eigenvalue for $B$ then ${{\beta }^{2}}$ is an eigenvalue for ${{B}^{2}}$, so we have ${{\beta }^{2}}=\lambda $.   However, we assumed that $(\frac{\lambda }{p})=-1$. Thus we obtain the desirable  contradiction.    The eigenvalue $\beta$ has geometrical dimension 1, because of in oppositive case geometrical $\dim\beta =2$ (dimension of eigenvector space of $\beta$), then we get that $J_{B}^{2}$ is only scalar matrix $B$.

The proof is fully complete.
\end{proof}



\begin{exm}
A  sufficiency  of the condition $(\frac{\lambda }{p})=1$ for $\exists $ $B:\,\,{{B}^{2}}=A$, where $A \sim {{J}_{A}}$ of size $2\times 2$ with one eigenvalue corresponding to one eigenvector is given by following matrix from $S{{L}_{2}}(R)$:

${{J}_{A}}=\left( \begin {array}{cc}
   1\,\,\, & 1 \\
  0\,\,\, & 1 \\
\end{array} \right)$ then $B={{\left( \begin{array}{cc}
  \mu \,\,\,  & 1 \\
  0\,\,\,  & \mu  \\
\end{array} \right)}^{2}}=\left( \begin {array}{cc}
   {{\mu }^{2}}\,\,\, & 2\mu  \\
  0\,\,\,\,\,\,  & {{\mu }^{2}} \\
\end{array} \right),\,\,\,\,\mu =\pm \sqrt{1}$.

This confirms Theorem \ref{CriterionPSL}.
Choosing the base for $B$ to $A$ be in Jordan form (in Jordan base):  $UB{{U}^{-1}}$ we obtain

\[\left( \begin{array}{cc}
   \frac{\mu }{2}\,\,\,\, & 1 \\
  0\,\,\,\,\,\, & \frac{\mu }{2} \\
\end{array} \right)={{J}_{B}}.\]

The last matrix is expressed by conjugating of $B$ by a diagonal matrix.
\end{exm}

\begin{exm}
 Consider an example of a matrix showing the that the condition of the existence of a non-trivial Jordan block 2 by 2 is necessary in Theorem \ref{CriterionPSLJordan}.
Let
$A=\left( \begin{array}{cc}
  -1\,\,\,\,\,\, & 0 \\
  \,\,0\,\,\, & -1 \\
\end{array} \right)={{\rho }_{180}}$. This is a $-90$ degree rotation matrix. Then its square root $B\in S{{L}_{2}}(R)$ has form ${{\rho }_{90}}=\left( \begin{array}{cc}
   \,0\,\,\,\,& 1 \\
  -1\,\,& 0 \\
\end{array} \right)=B\in S{{L}_{2}}(R)$.  Note that $A$ is presented in the diagonal form. There are also roots $B_{1}^{{}}=\left( \begin{array}{cc}
   i\,\,\,\,\,\,\, & a \\
  0\,\,\, & -i \\
\end{array} \right)$ from $S{{L}_{2}}(C)$ as well as $B_{2}^{{}}=\left( \begin{array}{cc}
   -i\,\,\,\,\,\, & a \\
  \,0\,\,\,\,\,\, & i \\
\end{array} \right)$.
\end{exm}

Consider the commutative matrix algebra $A\lg \left[ A \right]=\left\langle E,A \right\rangle $ over ${{F}_{p}}$, where$A\in S{{L}_{2}}\left( {{F}_{p}} \right)$, A  is semisimple without multiple e.g. and $E$ is identity matrix.
Let us consider such formulation of Weederbern Artin theorem (Theorem W. A.). Each simple algebra is isomorphic to full matrix ring over algebra with division.

We refer to such formulation of theorem Wederburn Artin from Vanderwarden's book [1, page 372, 1979 year].
\textbf{Theorem W. A.}

 \begin{thm} Every simple algebra with 1 is isomorphic to a complete matrix ring ${{M}_{{{n}_{i}}}}[{{D}_{i}}]$ over a division algebra ${{D}_{i}}$.
\end{thm}

Recall well known statement.
\begin{prop}
 Every finite commutative division algebra ${{D}_{i}}$ is  finite field.
\end{prop}

The following Lemma \ref{IsomwithField} can be proved with using of Weederbern Artin  theorem in the following way.
As well known that every finite commutative division algebra is isomorphic to finite field.
(Note that the field is a division commutative algebra and complete matrix ${{M}_{2\times 2}}[{{F}_{p}}]$  ring over ${{F}_{p}}$ is algebra.)
But we need more specialized isomorphism making connection between $A$ and its e.g. as element of ${{F}_{p^2}}$ so we will construct it now.

Since ${{F}_{{{p}^{2}}}}$ is simple commutative algebra and elements of group  $A\lg \left[ A \right]=\left\langle E,A \right\rangle $ with additional zero matrix $2\times 2$ forms an commutative matrix ring then according to Weederbern Artin theorem we have $A\left[ M \right]\simeq {{F}_{{{p}^{2}}}}$.  Kelly Hamilton theorem yields that dimension of such algebra $A\lg \left[ A \right]$ is 2 because of characteristic polynomial $\chi_A(x)$ of $A$ has degree 2.

To give a self-contained proof, we propose the following Lemma.

Firstly we consider case $A\in S{{L}_{2}}({{F}_{p}})$ but after our reasoning can be easy extended on the group $G{{L}_{2}}({{F}_{p}})$ because the property property of the $\det A=1$ will not be involved in our reasoning.

\begin{lem}\label{IsomwithField}
 The matrix algebra $Alg [A]=\left\langle E,\,\,A \right\rangle \simeq {{F}_{{{p}^{2}}}}$.
\end{lem}

\begin{proof}
We show that algebra $Alg \left[ M \right]=\left\langle E,A \right\rangle $ is isomorphic to finite field ${{F}_{{{p}^{2}}}}$.
As well-known from Galois theory, a quadratic extension of $F_p$
can be constructed by involving of any external element.
As well-known from Galois theory, a quadratic extension of ${{F}_{p}}$ can be constructed by involving of any external element $g\in {{F}_{{{p}^{2}}}}\backslash {{F}_{p}}$  relatively to ${{F}_{p}}$. We denote this element by $i$, in particular, for $p=4m+3$ it may be an element satisfying the relation ${{i}^{2}}=-1$.
Note that the matrix of the rotation by 90 degrees, that is a matrix
$$I:= \left( \begin{array}{cc}
  0\,\,\,\,&1 \\
 -1\,\,&0 \\
 \end{array} \right) ={{\rho }_{90}}$$
 satisfies this relation and can used as an example of matrix $A$. In case when $p=4m+3$  such matrix $J:\,\,\varphi (J)=j$,  ${{j}^{2}}=-1$ exists too.

Obviously $\det A=1$, that's why $A\in S{{L}_{2}}({{F}_{p}})$ and ${{\mu }_{A}}(x)$ is irreducible.
We define mapping $\varphi :\,\,{{y}_{1}}A+{{x}_{1}}E\,\to ae+b\lambda ;\,\,\,\,a,b\in {{\text{F}}_{p}}$.
The mapping $\varphi $ can be more broadly described, in $S{{L}_{2}}[{{F}_{p}}]$ such a way that a matrix $A$   is found such that ${{A}^{2}}=-E$, then its e.g. $\lambda $ is assigned to it in the field ${{F}_{{{p}^{2}}}}$, while $\lambda \in {{F}_{{{p}^{2}}}}\backslash {{F}_{p}}$.
$\varphi :\,\,{{y}_{1}}A+{{x}_{1}}E\,\to ae+b\lambda ;\,\,\,\,a,b\in {{\text{F}}_{p}}$.
According to assumption of Lemma the matrix $A$ is semisimple and has not multiple eigenvalues (e.g.) which are not squares in ${{\text{F}}_{p}}$, so ${{\chi }_{A}}\left( x \right)$  is irreducible because of definition of semisimple matrix and condition ${{\lambda }_{1}}\ne {{\lambda }_{2}}$. According to Lemma about Frobenius automorphism its eigenvalues are conjugated in ${{\text{F}}_{{{p}^{2}}}}$.
The method of constructing of $\sqrt{A}$ is the following. Having isomorphism $A\lg \left[ A \right]=\left\langle E,A \right\rangle \simeq {{\text{F}}_{{{p}^{2}}}}$ we set a correspondence $\lambda \leftrightarrow A$ and correspondence between groups operations in ${{\text{F}}_{{{p}^{2}}}}$ and $A\lg \left[ A \right]$.  Therefore solving equation ${{\left( x+\lambda y \right)}^{2}}=\lambda $ relatively coefficients $x,\,\,y\in {{\text{F}}_{p}}$ we obtain coefficients for expression for $\sqrt{A}$ i.e. $\sqrt{A}=x+Ay$.
To prove the isomorphism, we establish a bijection between the generators of the algebra $A\lg \left[ A \right]=\left\langle E,A \right\rangle $ and the field ${{F}_{{{p}^{2}}}}$. It is necessary to establish in more detail that $A\leftrightarrow \lambda $ and  $E\leftrightarrow e$ also  the correspondence between the neutral elements of both structures, i.e. $\varphi \left( {\bar{0}} \right)=0$ where $0$ is the zero matrix.
To complete proof, it remains to show that the kernel of this homomorphism $\varphi $ is trivial. To do this, we show that among the elements of the algebra there are no identical ones. The surjectivity of $\varphi $ is obvious.  From the opposite,  we assume ${{y}_{1}}A+{{x}_{1}}E={{y}_{2}}A+{{x}_{2}}E$, ${{x}_{i}},\,{{y}_{i}}\in {{F}_{p}}$.  Then ${{y}_{1}}A+{{x}_{1}}E={{y}_{2}}A+{{x}_{2}}E$ it yields that
$\left( {{y}_{1}}-{{y}_{2}} \right)E=\left( {{x}_{1}}-{{x}_{2}} \right)A$,
which is impossible since the characteristic polynomial of the matrix $A$ is irreducible but  the characteristic polynomial of the identity matrix is reducible. Therefore, our algebra $A\lg \left[ A \right]$ is isomorphic to the completely linear space of linear polynomials from $E$ and $A$.
In the similar way we prove that polynomial of form $xe+y\lambda $ where $x,\,\,y\in {{F}_{p}}$ do not repeat. The proof is based on oppositive assumption about coinciding ${{x}_{1}}e+{{y}_{1}}\lambda ={{x}_{2}}e+{{y}_{2}}\lambda$ of polynomial with different coefficients. Then equality ${{x}_{1}}e+{{y}_{1}}\lambda ={{x}_{2}}e+{{y}_{2}}\lambda$  implies that $\left( {{y}_{1}}-{{y}_{2}} \right)\lambda =\left( {{x}_{1}}-{{x}_{2}} \right)e$ i.e. ${{y}_{1}}={{y}_{2}}$ and ${{x}_{1}}={{x}_{2}}$ that contradicts to assumption.
\end{proof}

\begin{thm}
If a matrix $A\in G{{L}_{2}}({{F}_{p}})$ is semisimple with different eigenvalues but with the same quadratic, then $\sqrt{A}\in G{{L}_{2}}({{F}_{p}})$ iff at least one an eigenvalue ${{\lambda }_{i}}$, $ i\in \{1,2\}$  of $A$ satisfies:

 \begin{equation*}
(\frac{{{\lambda }_{i}}}{p})=1
 \, \, in \, \, the \, \, square \, \, extention \, \, that \, \, is \, \, {{F}_{p^2}}.
 \end{equation*}
\end{thm}

\begin{proof}
Firstly we consider most complex and interesting case when $A$ is not diagonalizable, then ${{\chi }_{A}}\left( x \right)$ is irreducible over  ${{F}_{p}}$. By assumption, the matrix is semisimple and its characteristic polynomial is irreducible. So root $\lambda $ of ${{\chi }_{A}}(x)$ belongs to the quadratic extension of the field ${{F}_{p}}$. Since each element of ${{F}_{{{p}^{2}}}}$ can be presented in form $a+b\lambda ,\,\,\,a,b\in {{F}_{p}}$, then we can construct mapping of matrix algebra generators $E$ and $A$  in generators of ${{F}_{{{p}^{2}}}}$ and apply aforementioned Lemma \ref{IsomwithField} about isomorphism establish correspondence between property be square in ${{\text{F}}_{{{p}^{2}}}}$ and in  $Alg \left[ A \right]=\left\langle E,A \right\rangle $. If one e.v. ${{\lambda}_{i}}$ is square in $F_p^2$ then so is second e.v. because of they a conjugated as roots of characteristic polynomial $\chi_A(x)$ by theorem about Frobenius automorphism (Frobenius endomorphism in perfect field became to be automorphism).
\end{proof}

\begin{exm}
Consider matrix $A=-E$, where $E$ is identity matrix in $S{{L}_{2}}({{F}_{3}})$satisfying conditions of Theorem 3 because of $(\frac{-1}{9})=1$ in ${{\text{F}}_{9}}$.
Then the matrix $\left( \begin{array}{cc}
   \,\,0\,\,\,\,2 \\
  -2\,\,0 \\
\end{array} \right)\in S{{L}_{2}}({{F}_{3}})$  is square root for $A$. Indeed ${{I}^{2}}=-E$.

Another root of equation ${{X}^{2}}=A$ is matrix of elliptic type realizing rotation on 90 ${{\rho }_{90}}=\left( \begin{array}{cc}
  \,\,\, 0\,\,\,\,1 \\
  -1\,\,0 \\
\end{array} \right)=I$, because of ${{I}^{2}}=-E$. The matrix $2I$ is the square in $G{{L}_{2}}\left( {{F}_{3}} \right)$ because of existing such an element ${{\left( \begin{array}{cc}
   \,\,\,1\,\,\, 1 \\
  -1\,\,\,\,1 \\
\end{array} \right)^{2}}}=2\left(\begin{array}{cc}
   \,\,  \ 0\,\,\,\,1 \\
  -1\,\,0 \\
\end{array} \right)=2I$.

\end{exm}

\begin{thm} \label{diagonal form}  If matrix $A\in S{{L}_{2}}({{F}_{p}})$ ($A \in GL({F}_{p})$) admits diagonal Jordan form over ${{\text{F}}_{p}}$, then $\sqrt{A}\in S{{L}_{2}}({{F}_{p}})$ ($GL({F}_{p})$) if and only  if $(\frac{{{\lambda }_{1}}}{p})=1$ and $(\frac{{{\lambda }_{2}}}{p})=1$ over ${{\text{F}}_{p}}$.
\end{thm}

\begin{proof}
From condition $(\frac{{{\lambda }_{1}}}{p})=1$ and $(\frac{{{\lambda }_{2}}}{p})=1$ it is followed, that ${{\mu }_{A}}\left( x \right)$ is reduced over ${{\text{F}}_{p}}$. Therefore why ${{\mu }_{1}},\,\,{{\mu }_{2}}\in {{\text{F}}_{p}}$ exist ${{\mu }_{B}}\left( x \right)$ over ${{\text{F}}_{p}}$ exists for matrix $B:\,\,{{B}^{2}}=A$. Assume that  $(\frac{{{\lambda }_{1}}}{p})=-1,\,\,(\frac{{{\lambda }_{2}}}{p})=-1$ prove, that while $\sqrt{A}\notin S{{L}_{2}}({{F}_{p}})$. We use proof by contradiction. Let  $(\frac{{{\lambda }_{1}}}{p})=-1,\,\,(\frac{{{\lambda }_{2}}}{p})=-1$  therefore roots from eigenvalues ${{\lambda }_{1}},\,\,{{\lambda }_{2}}$ in general belongs to ${{\text{F}}_{{{p}^{2}}}}$ while its roots ${{\mu }_{1}},\,\,{{\mu }_{2}}$ is not conjugated as roots from different values of ${{\lambda }_{1}},\,\,{{\lambda }_{2}}$.

Let's find minimal polynomial for $B=\sqrt{A}$.
    Minimal polynomial of matrix $B$ is ${{\mu }_{B}}(x)={{x}^{2}}-bx+c$ and it has different roots ${{\mu }_{1}},\,\,{{\mu }_{2}}$, where ${{\mu }_{1}}+{{\mu }_{2}}=Tr(B)=b$ и $\det B={{\mu }_{1}}{{\mu }_{2}}$. From the existence of diagonal representation for $A$ reducibility of ${{\mu }_{A}}\left( x \right)$ follows. From the reducibility of ${{\mu }_{A}}\left( x \right)$ over ${{\text{F}}_{p}}$ and the fact that ${{\lambda }_{1}}\ne {{\lambda }_{2}}$ follows ${{\mu }_{1}},\,\,{{\mu }_{2}}$ is not  conjugated as the roots of different values of ${{\lambda }_{1}},\,\,{{\lambda }_{2}}$ and it is obvious that $\mu _{1}^{2}\,\ne \,\mu _{2}^{2}$. But the root ${{\mu }_{1}}$ is conjugated with $-{{\mu }_{1}}$ and ${{\mu }_{1}}\in {{\text{F}}_{{{p}^{2}}}}\backslash {{\text{F}}_{p}}$. But $-{{\mu }_{1}}$ is also a root, since ${{\left( \pm {{\mu }_{1}} \right)}^{2}}={{\lambda }_{1}}$ therefore it can be the root for ${{\mu }_{B}}(x)$. Similar situation is with root ${{\mu }_{2}}$ and $-{{\mu }_{2}}$. Therefore, we indicated as many as 4 roots for ${{\mu }_{B}}(x)$ but $B\in S{{L}_{2}}({{F}_{{{p}^{2}}}})$ therefore $\deg \left( {{\mu }_{B}}(x) \right)=2$. This contradiction arise from the assumption that$\sqrt{A}\in S{{L}_{2}}({{F}_{p}})$ on condition $(\frac{{{\lambda }_{1}}}{p})=-1,\,\,(\frac{{{\lambda }_{2}}}{p})=-1$.
\end{proof}

\begin{rem} \label{diagonal form and 2 non quadratic residue}
 If $A\in S{{L}_{2}}({{F}_{p}})$ admits diagonal Jordan form over ${{\text{F}}_{p}}$ and $(\frac{{{\lambda }_{1}}}{p})=-1,\,\,(\frac{{{\lambda }_{2}}}{p})=1$,  then such case does not give the existence of solution of ${{X}^{2}}=A$ in $S{{L}_{2}}({{F}_{p}})$.
\end{rem}

\begin{proof}  The condition  $(\frac{{{\lambda }_{1}}}{p})=-1$ means, that $\sqrt{{{\lambda }_{1}}}={{\beta }_{1}}\in {{\text{F}}_{{{p}^{2}}}}\backslash {{\text{F}}_{p}}$  and simultaneously $\sqrt{{{\lambda }_{2}}}={{\beta }_{2}}\in {{\text{F}}_{p}}$,  therefore ${{\beta }_{1}}+{{\beta }_{2}}=Tr(B)\notin {{F}_{p}}$. This implies non-existing of ${{\mu }_{B}}\left( x \right)$ over ${{\text{F}}_{p}}$.
\end{proof}

\begin{thm}
(Theorem 2.) If the matrix $A\in S{{L}_{2}}({{F}_{p}})$ is semisimple and diagonalizable over${{\text{F}}_{p}}$ and $(\frac{{{\lambda }_{1}}}{p})=(\frac{{{\lambda }_{2}}}{p})=-1$, then  for existing a root $\sqrt{A}$, it is necessary and sufficient that $A$ is similar to scalar matrix $D$.
\end{thm}

\begin{proof}
From the facts that  $(\frac{{{\lambda }_{1}}}{p})=(\frac{{{\lambda }_{2}}}{p})=-1$ and the square of diagonal matrix is again the diagonal matrix  follows the existence of root only in the off diagonal form, therefore we must find the solution $M$ among the set of non-diagonalizable matrices.
$ D=\left( \begin{array}{cc}
   & {{d}_{1}}\,\,\,0 \\
 & 0\,\,\,{{d}_{2}} \\
 \end{array} \right) $
is the diagonal representation of matrix $A$, and let
 \begin {equation}\label{D}
 D={{M}^{2}},
\end {equation}

where $M\in S{{L}_{2}}({{F}_{p}})$, Because of $(\frac{{{d}_{1}}}{p})=(\frac{{{d}_{2}}}{p})=-1$ there is a  root in non-diagonal form. Also we note that there is a conjugation matrix $X$,

$$X=\left( \begin{array}{cc}
  & m_{11}^{-1}\,\,\,0 \\
 & 0\,\,\,\,\,\,m_{21}^{-1} \\
\end{array}  \right),$$  transforming  $M$  to $\tilde{M}$, where

$$\tilde{M}=~\left( \begin{array}{cc}
  & {{m}_{11}}\,\,\,\,1 \\
 & {{m}_{21}}\,\,{{m}_{22}} \\
\end{array} \right)$$.

Let's transform the equality $D={{M}^{2}}$ into $XD{{X}^{-1}}=XM{{X}^{-1}}XM{{X}^{-1}}$, where $XM{{X}^{-1}}=\tilde{M}$.
Note that $D$ and $XD{{X}^{-1}}$ have identical eigenvalues. Therefore we can solve the equation \eqref{D} for $XD{{X}^{-1}}$.  Let's consider matrix equation  $D={{M}^{2}}$ ,  let's transform it by conjugation
$D=XD{{X}^{-1}}=XM{{X}^{-1}}\cdot XM{{X}^{-1}}=\tilde{M}\tilde{M}={{\tilde{M}}^{2}}$.

wherein $M=\left( \begin{array}{cc}
  & {{m}_{11}}\,\,\,\,{{m}_{12}} \\
 & {{m}_{21}}\,\,{{m}_{22}} \\
\end{array} \right)$
and
$XM{{X}^{-1}}=\left( \begin{array}{cc}
  & {{m}_{11}}\,\,\,\,1 \\
 & {{m}_{21}}\,\,{{m}_{22}} \\
\end{array} \right),$

for obtaining such conjugation matrix we use
$X=\left( \begin{array}{cc}
  & m_{11}^{-1}\,\,\,0 \\
 & 0\,\,\,\,\,\,m_{21}^{-1} \\
\end{array} \right).$

Note that since $D$ is a diagonal matrix, then it belongs to the commutative subgroup of diagonal matrices from  $S{{L}_{2}}({{F}_{p}})$, lets denote it as  $DS{{L}_{2}}({{F}_{p}})$.
Therefore and  $~XD{{X}^{-1}}$  is also a diagonal matrix. Moreover, due to the commutativity of the field ${{F}_{p}}$ we have $~XD{{X}^{-1}}=D.$
Now let's solve the matrix equation  for the reduced $\tilde{M}$

 \begin {equation}\label{XDX}
 D=~XD{{X}^{-1}}=(XM{{X}^{-1}})(XM{{X}^{-1}})={{\tilde{M}}^{2}},
\end {equation}

Note that equations \eqref{XDX} and \eqref{D} are equivalent since they are obtained by similarity transformations.

Note that equations (2) и (1) are equivalent since they are obtained by similarity transformations.
Let's write down the equation $\left( \begin{array}{cc}
  & {{d}_{1}}\,\,\,0 \\
 & 0\,\,\,{{d}_{2}} \\
\end{array} \right)={{\left( \begin{array}{cc}
  & {{m}_{11}}\,\,\,\,1 \\
 & {{m}_{21}}\,\,\,{{m}_{22}} \\
\end{array} \right)}^{2}}={{\tilde{M}}^{2}}$.

Thence we obtain the system of equations
$$\left\{ \begin{array}{cc}
  {{m}_{21}}+m_{11}^{2}={{d}_{1}} \\
 {{m}_{21}}+m_{22}^{2}={{d}_{2}} \\
 {{m}_{11}}+{{m}_{22}}=0, \\
\end{array} \right.$$
by substitution from the equation 3) ${{m}_{22}}=-{{m}_{11}}$ into equations 1) and 2) we obtain \[2){{m}_{21}}+m_{22}^{2}={{d}_{2}} \Rightarrow \]\[ \Rightarrow {{m}_{21}}+{{(-m_{11}^{{}})}^{2}}={{d}_{2}}\] also we take into consideration equation 1) $m_{11}^{2}+{{m}_{21}}={{d}_{1}}$. Thence ${{d}_{1}}={{d}_{2}}$ or more conveniently $d={{d}_{1}}={{d}_{2}}$. Wherein $d$ doesn't have to be a quadratic residue. Therefore the conditions $(\frac{d}{p})=-1$ of theorem are met.
\end{proof}

The previous theorem can be proved in following way using finite field equations.
\begin{rem}
 If the matrix $A \in SL_2 (F_p)$   is semisimple and diagonalizable over $F_p$  and $(\frac{{{\lambda }_{1}}}{p}) = (\frac{{{\lambda }_{2}}}{p}) = -1$, then  for existing $\sqrt{A} \in SL_2 (F_p)$, it is necessary and sufficient that $A$ is similar to scalar matrix.
\end{rem}

\begin{proof}  (Way 2.) Let according to condition both e.v. $(\frac{{{\lambda }_{1}}}{p}) = (\frac{{{\lambda }_{2}}}{p}) \neq -1$  and ${\lambda }_{1} \neq {\lambda }_{2}$.

 If we assume, that  equation $X^2=A$ is solvable over $F_p$  then e.v. $\beta_1, \beta_2$  of $B$  are such that $\beta_1, \beta_2 \in F_p^2 \backslash F_p$
because of $\alpha$  is non-square over $F_p$. Since we assume existing of root from $A$  in $A \in SL_2 (F_p)$, then trace of  $B$  (denoted by $ $) is the following   $Tr_B=\beta_1+\beta_2$.  Consider expression $\beta_1$  by trace $\beta_1= t-\beta_2$    then its square ${\beta_1}^2= (t-\beta_2)^2$  give us  ${\beta_1}^2= t^2 -2t\beta_2 - \beta_2^2$. In the LHS of the equality we have that $\beta_1^2 \in F_p$   because $\beta_1$   is root from the element $\alpha_1 \in F_p$. But in RHS $t^2-2t\beta_2+\beta^2_2 \in F_p $ iff $t=0$  because of  $\beta_1, \beta_2 \in F_p^2 \backslash F_p $  as root from $\alpha_2$  which is not square in $F_p$. But we assume that $B$  is arbitrary diagonal matrix over $F_p$  that means $t\neq 0$. The obtained contradiction follows from our assumption that $A$  is not scalar matrix.
\end{proof}

Recall the following property about verbal width of commutator subgroup $G{'}$.
\begin{lem} \label{G^2}  The commutator subgroup is subgroup of $G^{2}$  i.e. $G{'}<G^{2}$.
\end{lem}

Indeed, an arbitrary commutator presented as product of squares. Let $a,\,\,b\in G$ and set that
$x=a,\,\,\,y={{a}^{-1}}ba,\,\,\,z={{a}^{-1}}{{b}^{-1}}$ then ${{x}^{2}}{{y}^{2}}{{z}^{2}}={{a}^{2}}{{({{a}^{-1}}ba)}^{2}}{{({{a}^{-1}}{{b}^{-1}})}^{2}}=ab{{a}^{-1}}{{b}^{-1}},\,\,$
in more detaile:
${{a}^{2}}{{({{a}^{-1}}ba)}^{2}}{{({{a}^{-1}}{{b}^{-1}})}^{2}}={{a}^{2}}{{a}^{-1}}ba\,{{a}^{-1}}ba\,\,{{a}^{-1}}{{b}^{-1}}{{a}^{-1}}{{b}^{-1}}=abb{{b}^{-1}}{{a}^{-1}}{{b}^{-1}}=\left[ a,b \right]$.
In such way we obtain all commutators and their products.
Thus, we generate by squares whole $G{{'}_{k}}$.



\begin{prop}
The following Mathieu's groups ${M}_{11}$, ${M}_{12}$, $M_{20}$, ${M}_{21}$ $M_{22}$, $M_{23}$ and $M_{24}$ have verbal width $(wid(G, S(G)))$ by square 2.
\end{prop}

\begin{proof}
Due to the exceptional  isomorphism ${{M}_{21}}\simeq PSL\left( 3,\text{ }4 \right)$  \cite{Doc}  we  can investigate prime group  $PSL\left( 3,\,4 \right)$. As well as the set  $S(PSL\left( 3,\,4 \right))$ generates whole group $PSL\left( 3,\,4 \right)$  the question  of verbal width remains relevant.  According to Lemma \ref{G^2} and previous Corollary it bounded by 3. The computation on GAP system give us answer $vw(M_{21})=2$.  The exceptional isomorphisms ${{M}_{12}}\simeq PSL\left( 2,\text{ 11} \right)$ and ${{M}_{21}}\simeq PSL\left( 3,\text{ }4 \right)$ gives us the possibility to investigate well known matrix group from $PSL\left( k,\text{ }n \right)$ and obtain result $vw(M_{12},S(M_{12}))=2$, $vw(M_{12},S(M_{21}))=2$ in analytical way. For verifying the verbal width for the rest of groups algebraic system GAP was applied.
\end{proof}


\section{Corollary }

If a matrix $A\in G{{L}_{2}}({{F}_{p}})$ is semisimple with different eigenvalues but with the same quadratic, then $\sqrt{A}\in G{{L}_{2}}({{F}_{p}})$ iff at least one an eigenvalue ${{\lambda }_{i}}$, $ i\in \{1,2\}$  of $A$ satisfies:

 \begin{equation*}
(\frac{{{\lambda }_{i}}}{p})=1
 \, \, in \, \, the \, \, algebraic \, \, extention \, \, of \, \, degree \, \, 2 \, \, that \, \, is \, \, {{F}_{p^2}}.
 \end{equation*}

Verbal width by $S(A_{n})$ of the group $ A_{n} $ equals to $2$ for $n>4$. Commutator width of $syl_2 A_{2^k}$ is 1 \cite{SkCommEur, SkAr}.



\begin{thebibliography}{99}

 \bibitem{Sar} \emph{Sucharit Sarkar.}  Commutators and squares in free group, {\it Algebra Geometry Topologe}, \textbf{4}б (1),  (2004), pp. 595-602. https://doi.org/10.2140/agt.2004.4.595

\bibitem{Lynd} \emph{R. C. Lyndon, M. F. Newman.} Commutators as products of squares, Proc.
Amer. Math. Soc. 39 (1973) 267-272 MR0314997

\bibitem{Amit} \emph{Amit Kulshrestha and Anupam Singh.} "Computing $n$-th roots in SL2 and Fibonacci polynomials"
{\it Proc. Indian Acad. Sci.} (Math. Sci.) (2020) 130:31 https://doi.org/10.1007/s12044-020-0559-8.

\bibitem{Liebeck} \emph{Martin Liebeck,  E. A. O'Brien,  Aner Shalev,  Pham Huu Tiep.} Products of squares in finite simple groups.
 Proceedings of the American Mathematical Society
Vol. 140, No. 1 (JANUARY 2012), pp. 21-33.



\bibitem{Dix} \emph{John D. Dixon, Brian Mortimer.}   Permutation groups.  \emph {Graduate texts in mathematics}; \textbf{ 163 } (1996),   P. 348.



\bibitem {Carm}  \emph{Robert Carmichael.} Introduction to the theory of Groups of Finite Order. Dover, 1956, 431. ISBN 0-486-60300-8.

\bibitem {Carm2} \emph{R. D. Carmichael.} Tactical Configurations of Rank Two.  American Journal of Mathematics,  Vol. 53, No. 1 (Jan., 1931), pp. 217-240.

\bibitem {aviv}  \emph{Aviv Rotbart.} Generator sets for the alternating group. Seminaire Lotharingien de Combinatoire 65 (2011), Article B65b.

\bibitem {CONRAD} \emph{Keith Conrad.} Generating sets. www.math.uconn.edu/~kconrad/blurbs/grouptheory/genset.pdf.

\bibitem {Chekh} \emph{A. R. Chekhlov.} "Intermediately fully invariant subgroups of abelian groups", Siberian Math. J., 58:5 (2017), 907-914.

\bibitem {ChekhDanch} \emph{ A. R. Chekhlov, P. V. Danchev,} "The strongly invariant extending property for abelian groups", Quaest. Math., 42:8 (2019), 997-1017.

\bibitem {Seg} \emph{Segal, D.}  Notes on Verbal Width in Groups (London Mathematical Society Lecture Note Series). Cambridge: Cambridge University Press. (2009). doi:10.1017/CBO9781139107082

\bibitem {Cameron} \emph{Peter J. Cameron and Philippe Cara.} Independent generating sets and geometries for sym
metric groups. Journal of Algebra, 258(2):641 – 650, 2002.



\bibitem {O.Ore} \emph{O. Ore.} Some remarks on commutators, Proc. Amer. Math. Soc. 2 (1951), p. 307-314.

\bibitem{Ore} \emph{Martin Liebeck,  E. A. O'Brien,  Aner Shalev,  Pham Huu Tiep.} The Ore conjecture. Journal of the European Mathematical Society 12(4) January 2010  p.939-1008
\bibitem{Doc} \emph{V. V. Dotsenko}
 Notes on exceptional isomorphisms.  Mathematical education, ser.
 3, no. 12, 2008 (81–93).

\bibitem{Rom} \emph{V. A. Roman'kov}. The commutator width of some relatively free lie algebras and nilpotent groups. Siberian Mathematical Journal volume 57, pages 679-695 (2016).

\bibitem{SkCommEur}
\emph{ Ruslan V. Skuratovskii}
On commutator subgroups of Sylow 2-subgroups of the alternating group, and the commutator width in wreath products.
European Journal of Mathematics (2021), volume 7, pages 353–373.

\bibitem{SkRel} \emph{ Drozd, Yu.A.,  R.~V. Skuratovskii,}   Generators and relations for wreath products.
Ukr Math J. (2008), vol. 60. Issue 7, pp. 1168-1171.

\bibitem{SkAr} \emph{Ruslan V. Skuratovskii}, Minimal generating systems and properties of Sylow 2-subgroups of alternating group. Source [ arXiv:1607.04855], https://arxiv.org/pdf/1607.04855.pdf

\bibitem{Por} \emph{Poroshenko E. N.}, "Commutator width of elements in a free metabelian Lie algebra," Algebra and Logic, 53, No. 5, 377-396 (2014).

\bibitem{Nik} Nikolov, N.: On the commutator width of perfect groups. Bull. London Math. Soc. 36(1), 30–36 (2004)

\bibitem{Gur} Guralnick, R.M.: Commutators and wreath products. In: Lewis, M.L., et al. (eds.) Character Theory of Finite Groups. Contemporary Mathematics, vol. 524, pp. 79–82.
 American Mathematical Society, Providence (2010)


\bibitem{SkAbst} Skuratovskii R.V. On the verbal width in the alternating group $A_n$ and Matieu groups1 // Modern problems in mathematics and its applications International (53rd National) Youth School-Conference, 2022, Yekaterinburg, 2022 from January 31 to February 4 section Group theory, pp. 1-2. https://sopromat.imm.uran.ru.

\bibitem{SkAbst21} Skuratovskii R.V. On the verbal width in the alternating group $A_n$ and Matieu groups. // International Algebraic Conference, dedicated to the 90th anniversary of the birth of A.I. Starostin. Book of Abstracts. 05 October 2021 - 09 October 2021. pp. 107-108.

\bibitem{Sachk} Sachkov, V.N.,  Combinatorial methods in discrete Mathematics. Encyclopedia of mathematics and its applications 55. Cambridge Press. 2008. P. 305.

\bibitem{Linear}
Nering, Evar D., Linear Algebra and Matrix Theory (2nd ed.), (1970), New York: Wiley, LCCN 76091646.

\bibitem{LinearShpringer}
Jцrg LiesenVolker Mehrmann. Linear Algebra.Springer Undergraduate Mathematics Series. Springer International Publishing Switzerland 2015 (2015). DOI https://doi.org/10.1007/978-3-319

\bibitem{GM}
Jane Gilman. "ADJOINING ROOTS AND RATIONAL POWERS OF
GENERATORS IN PSL(2,R) AND DISCRETENESS." [source: arXiv:1705.03539v2 [math.GR] 30 Nov 2017].


\bibitem{Pell} \emph{ Saadet Arslan, Fikri Kцken.}
The Pell and Pell-Lucas Numbers
via Square Roots of Matrices. Journal of Informatics and Mathematical Sciences
Vol. 8, No. 3, pp. 159–166, 2016.

\bibitem{Kunyav} \emph{ Bandman T., Greuel G.-M., Grunewald F., Kunyavskii B., Pfister G., Plotkin E.} Identities for finite solvable groups and equations in finite simple groups, Compos. Math., 2006, 142(3), 734–764.


\bibitem{Kunyav2} \emph{ Bandman T., Kunyavskii B.} Criteria for equidistribution of solutions of word equations in SL (2), J. Algebra, 2013, 382, 282–302.

\bibitem{Klyach} \emph{Klyachko Anton A., Baranov D. V.} Economical adjunction of square roots to groups. Sib. math. journal, Volume 53 (2012), Number 2, pp. 250-257.

\end{thebibliography}
\end{document}